\documentclass[reqno, 12pt,a4letter]{amsart}
\usepackage{amsmath,amsxtra,amssymb,latexsym,amscd,amsthm}
\usepackage[utf8]{inputenc}
\usepackage[mathscr]{euscript}
\usepackage{mathrsfs}
\usepackage[english]{babel}
\usepackage{color}
\usepackage[active]{srcltx}
\usepackage{enumerate}
\usepackage[T1]{fontenc}
\usepackage{tikz}
\usetikzlibrary{calc}

\newtheorem*{H-L}{Theorem Henry--L{\^e}}
\newtheorem*{O-W}{Theorem O'Shea--Wilson}
\newtheorem {theorem}{Theorem}[section]

\newtheorem {corollary}[theorem]{Corollary}
\newtheorem {lemma}[theorem]{Lemma}
\newtheorem {example}[theorem]{Example}
\newtheorem {definition}[theorem]{Definition}
\newtheorem {remark}[theorem]{Remark}

\newcommand{\dist}{{\rm dist}}

\newcommand{\proj}{{\rm pr}}

\newcommand{\rank}{{\rm rank}}
\newcommand{\R}{\mathbb R}
\newcommand{\B}{\mathbb B}
\newcommand{\C}{\mathbb C}

\newcommand{\N}{\mathbb N}

\def\ees{{\accent"5E e}\kern-.385em\raise.2ex\hbox{\char'23}\kern-.08em}
\def\EES{{\accent"5E e}\kern-.5em\raise.8ex\hbox{\char'23 }}
\def\ow{o\kern-.42em\raise.82ex\hbox{\vrule width .12em height .0ex depth .075ex \kern-0.16em \char'56}\kern-.07em}
\def\OW{o\kern-.460em\raise1.36ex\hbox{
\vrule width .13em height .0ex depth .075ex \kern-0.16em
\char'56}\kern-.07em}
\def\DD{D\kern-.7em\raise0.4ex\hbox{\char '55}\kern.33em}

\title{Limits of tangent spaces to definable sets}

\author{S\~i Ti\d{\^e}p \DD inh$^\dagger$}
\address{Institute of Mathematics, VAST, 18 Hoang Quoc Viet Road, Cau Giay District 10307, Hanoi, Vietnam and Institute of Mathematics, Polish Academy of Sciences, \'Sniadeckich 8, 00-656 Warsaw, Poland}
\email{dstiep@math.ac.vn}

\author{Olivier Le Gal}
\address{Universit{\' e} de Savoie Mont Blanc, Laboratoire de Math{\' e}matiques, B{\^ a}timent Chablais, Campus Scientifique, 73376 Le Bourget-du-Lac Cedex, France}
\email{Olivier.Le-Gal@univ-savoie.fr, Olivier.Le-Gal@univ-smb.fr }

\author{Ti\'{\^e}n-S\OW n Ph\d{a}m$^\ddagger$}
\address{Department of Mathematics, Dalat University, 1 Phu Dong Thien Vuong, Dalat, Vietnam}
\email{sonpt@dlu.edu.vn}



\date{\today}

\begin{document}

\begin{abstract} We study the set of tangent limits at a given point to a set definable in any o-minimal structure by characterizing the set of ``exceptional rays'' in the tangent cone to the set at that point and investigating the set of tangent limits along these rays. Several criteria for determining ``exceptional rays'' will be given. 
The main results of the paper generalize, to the o-minimal setting and to arbitrary dimension, the main results of~\cite{OShea2004} which deals with algebraic surfaces in $\mathbb R^3$. 
\end{abstract}

\maketitle

\pagestyle{plain}

\section{Introduction}
One of the ways to study singular varieties is to investigate their tangent cones and limits of tangent spaces, which was initialized by Whitney in the 1960s~\cite{Whitney1965-1,Whitney1965-2}. 
Given a variety $X\subset \mathbb K^n$ of (pure) dimension $d$ ($\mathbb K=\mathbb R \text{ or }\mathbb C$, varieties nature will be specified), one constructs the {\em Nash blow-up} $\mathcal N(X)$ of $X$, made of the closure of the tangent bundle of its regular part $X_{reg}$ (we say that $x\in X$ is regular if $X$ is $C^1$ at $x$) and studies the {\em Nash fiber} $\mathcal N(X)_x$ of the bundle $\mathcal N(X)$ over $x\in X$. Namely, if $\mathbb G(d,n)$ denotes the Grassmannian of the $d$-dimensional linear subspaces of $\mathbb K^n$ and $T_xX$ is the space tangent to $X$ at $x$, then
$$\mathcal N(X) = \overline{\{(x,P)\in X_{reg} \times \mathbb G(d,n):\ P = T_x X_{reg}\}},$$
so $(x,P)$ belongs to $\mathcal N(X)$ if there exists a sequence $x^k\in X_{reg}$ approaching $x$ with $T_{x^k}X\to P$. If $x$ is a regular point, the Nash fiber $\mathcal N(X)_x$ is reduced to the tangent space to $X$ at $x$, but for singular $x$, this Nash fiber contains all limits at $x$ of the spaces tangent to $X_{reg}$, then carries informations on the singularity germ. 

Nash fibers are better analyzed together with an additional data which keeps track of the direction along which limits are taken. For this, set
$$\mathcal N'(X) = \overline{ \{ (x,t,v,P)\in X_{reg}\times\mathbb K_+ \times\mathbb K^n \times \mathbb G(d,n):\ x+tv \in X_{reg},\ P= T_{x+tv} X_{reg}\} },$$
where $\mathbb K_+=\mathbb R_+:=(0,+\infty)$ if $\mathbb K=\mathbb R$ and $\mathbb K_+=\mathbb C\setminus\{0\}$ if $\mathbb K=\mathbb C.$ 
The fiber $\mathcal N'(X)_{(x,0)}$ of $\mathcal N'(X)$ for $x\in X$ and $t=0$ gives again the Nash fiber $\mathcal N(X)_x$ when projected on the $P$ coordinate, while its projection on the $v$ coordinate is the tangent (semi)cone to $X$ at $x$. 
Most importantly, $\mathcal N'(X)_{(x,0)}$ connects a plane in the Nash fiber to the direction of the tangent (semi)cone it comes from. 
Indeed, $(v,P)\in\mathcal N'(X)_{(x,0)}$ if there exists a sequence $x^k\in X$ approaching $x$ with simultaneously $T_{x^k}X\to P$ and $t_k(x^k-x)\to v$ for some $t_k\in\mathbb K_+$. 
For fixed $x$ and for each $v\in \mathbb K^n\setminus\{0\}$, the fiber of $\mathcal N'(X)$ over $(x,0,u)$ does not depend on the point $u$ in the ray $\ell:=\mathbb R_+ v$ if $\mathbb K=\mathbb R$ or the line $\ell:=\mathbb C v$ if $\mathbb K=\mathbb C$, and is denoted by $\mathcal N_{\ell}$ if $X$ and $x$ are clear from the context. We call $\mathcal N_\ell$, the {\em Nash fiber (over $x$) along $\ell$} or {\em the set of tangent limits to $X$ (at $x$) along $\ell.$} Studying a Nash fiber with respect to the tangent cone consists of  describing $\mathcal N_{\ell}$ with respect to $\ell$.

For complex varieties, these studies were carried out by Henry, L\^e and Teissier~\cite{Henry1975,Le1981,Le1988}. The case of complex algebraic surfaces is well described by~\cite{Henry1975} for isolated singularities and by~\cite{Le1981} in general, and the results can be stated as follows. 

\begin{H-L}[\cite{Henry1975,Le1981}]
Let $X\subset \mathbb C^3$ be a complex algebraic surface containing the origin $0\in \C^3$. Denote by $\mathcal C$ the tangent cone to $X$ at $0$. Then there exists a cone $\mathcal E\subset \mathcal C$ consisting of finitely many lines $\ell_1,\dots,\ell_r$ (called exceptional lines) such that:
\begin{itemize}
\item For $\ell\subset\mathcal C$ such that $\ell\not\subset\mathcal E$, the set $\mathcal N_{\ell}$ is reduced to one plane, which is the (common) tangent plane to $\mathcal C$ at a non zero point of $\ell$.
\item For $\ell\subset\mathcal E$, we have $\mathcal N_{\ell}=\{P\in \mathbb G(2,3):\ \ell\subset P\}$, i.e., $\mathcal N_{\ell}$ contains the whole pencil of planes containing $\ell$.
\item Singular lines of $\mathcal C$ are exceptional, i.e., $\mathcal C_{sing}\subset \mathcal E$.
\end{itemize}
\end{H-L}

In the real setting, to the best of our knowledge, only algebraic surfaces in $\mathbb R^3$ have already been considered elaborately, by O'Shea and Wilson in \cite{OShea2004}. Compared to the complex case, the structure of the set of tangent limits at a singular point of a real algebraic surface is more flexible, and characterizations of exceptional rays that coincide for complex varieties become inequivalent. In \cite{OShea2004}, a ray $\ell$ is said to be {\em exceptional} if $\mathcal N_{\ell}$ has positive dimension.  
The main results in~\cite{OShea2004} can be summarized as follows.

\begin{O-W}[\cite{OShea2004}]
Let $X\subset\mathbb R^3$ be a real algebraic surface containing the origin $0\in \R^3$. Let $\mathcal C$ be the tangent semicone to $X$ at $0$, and $\mathcal C'$ be the tangent semicone to the singular part $X_{sing}$ of $X$ at $0$. 
Then there exists a semicone $\mathcal E\subset \mathcal C$ not containing rays in $\mathcal C'$ and consisting of finitely many  rays $\ell_1,\dots,\ell_r$ (called exceptional rays),  such that:
\begin{itemize}
\item If $\ell\subset\mathcal C\setminus(\mathcal E\cup\mathcal C')$, then $\mathcal N_{\ell}$ is reduced to one plane, which is the (common) tangent plane to $\mathcal C$ at a non zero point of $\ell$.
\item If $\ell\subset \mathcal E$, then $\mathcal N_{\ell}$ is closed, connected and has dimension $1.$
\item If $\ell\subset \mathcal C_{sing}\setminus \mathcal C'$ is a singular ray of $\mathcal C$, then $\ell$ is exceptional, except, possibly, if the tangent semicone to $\mathcal C$ at non zero points of $\ell$ is a plane. In particular, if $\ell$ is an isolated ray or a boundary ray in $\mathcal C$, then $\ell$ is exceptional.
\end{itemize}
\end{O-W}

The authors of~\cite{OShea2004} remark that their results should have generalizations to higher dimension and/or codimension, but, as they fairly recognized, their methods are merely specific to algebraic surfaces in $\mathbb R^3$. This article aims to make this extension for arbitrary dimension and codimension. It happens moreover that the proofs can be made very general: we show that our results hold in the setting of an arbitrary o-minimal structure; in particular, the given description of Nash fibers is indeed not of an algebraic nature, but steams from the tame topology of the considered sets.

For the remainder of the paper, abusing of terminology, a ``semicone'' is called briefly a ``cone'' for short. 
If not mentioned otherwise, the term ``ray'' means ``open ray emanating from the origin $0\in\R^n$'', i.e., we consider only ``rays'' with the endpoint $0$ but $0$ is not included. 

Given an o-minimal expansion $\mathcal R$ of the field of real numbers, we call a set {\em definable} if it is definable in $\mathcal R$ with real parameters. We include in Section \ref{OM} a short introduction to o-minimality and refer to~\cite{Coste2000} or \cite{Dries1996} for further general references. Let $X\subset\R^n$ be a definable set with $x\in \overline X$.  The {\em geometric tangent semicone}, which we will call ``{\em tangent cone}'' for short from now on, of $X$ at $x$ is defined by
$$C_x X:=\left\{
\begin{array}{lll}
v\in\mathbb R^n : & \text{there are sequences } x^k\in X \text{ and } t_k\in(0,+\infty) \text{ such that}\\ 
&x^k\to x \text{ and } t_k (x^k-x)\to v \text{ as } k\to +\infty
\end{array}\right\}.$$
With no loss of generality, suppose that $x=0$ for the remainder of the paper. 
For short we set 
$$\mathcal C:=C_0 X \text{ and }\ \mathcal C':=C_0 (\overline{X})_{sing},$$  
where $(\overline{X})_{sing}$ is the singular part of the closure $\overline X$ of $X$. 
As mentioned previously, different subsets of $\mathcal C\setminus \mathcal C'$ might be considered as {\em exceptional} (like in Theorem O'Shea--Wilson, we exclude rays in $\mathcal C'$, both because such rays are certainly not ordinary, and because Nash fibers along these rays seem to be wilder as they contain the degeneracy at $0$ of singular Nash fibers at $x\neq 0$); namely: 
\begin{enumerate}
\item [(a)] non singular rays $\ell$ in $\mathcal C$ whose Nash fiber is not the tangent space $T_{v}\mathcal C$ to the tangent cone $\mathcal C$ at an arbitrary point $v$ in $\ell\colon $ $\mathcal N_{\ell}\neq \{T_{v}\mathcal C\}$; 

\item [(b)] rays $\ell$ whose Nash fiber is not a unique plane: $\# (\mathcal N_\ell)$ $>1$;
\item [(c)] rays $\ell$ whose Nash fiber has positive dimension: $\dim \mathcal N_\ell\geqslant 1$;
\item [(d)] rays $\ell$ that contains non zero critical values of the canonical projection 
 $$\mathop{pr}\colon \mathcal N'(X)_{(x=0,t=0)}\subset \mathcal C\times \mathcal N(X)_{x=0}\to\mathcal C,\ (v,P) \mapsto v.$$
($\mathcal N_{\ell}$ is precisely the (common) preimage of $v\in \ell$, by this projection). 
\end{enumerate}
We will focus mainly on the criterion (b) in this article, so we define the following set: 
\begin{equation}\label{E}
\mathcal E=\{v\in \mathbb R^n\setminus\{0\}\colon\ \ell:=\mathbb R_+ v\subset\mathcal C\setminus \mathcal C',\; \#(\mathcal N_\ell) > 1\}.
\end{equation}
The choice of criterion (b) as principal interest can be explained in light of our results as follows. We show that criteria (a) and (b) coincide for rays $\ell\not\subset \mathcal C_{sing}\cup\mathcal C'$ with $\dim_0 X= \dim_v\mathcal C$, where $v\in\ell$  and $\dim_0 X$ is the dimension of $X$ at $0$ (Theorem~\ref{Codim1}). For hypersurfaces, we show that (b) and (c) coincide (Theorem \ref{Connected}); they do not in full generality (Example~\ref{Codim2}), 
and we do not know if they coincide for any ray $\ell \not\subset \mathcal C_{sing}.$ Criterion (d) is not studied here since we follow~\cite{OShea2004}, while in view of Singularity Theory, it is a natural candidate. The projection $\mathop{pr}$ and its critical values emerge however here and there during the proofs, and we believe that (d)  deserves its own study.

It is noticeable that the criteria (a), (b) and (c) collapse when $\dim_0 X> \dim\mathcal C$. In order to deal with this situation, we also introduce
the following cone of rays $\ell$ whose Nash fiber has a tangent limit not containing the plane tangent to $\mathcal C$ along $\ell$:
\begin{equation}\label{E'}
\mathcal E'=\{v\in \mathbb R^n\setminus\{0\}:\ \ell:=\mathbb R_+ v\subset \mathcal C\setminus \mathcal C_{sing},\text{ there exists } P\in\mathcal N_\ell \text{ such that } T_v\mathcal C\not\subset P\},
\end{equation}
so $\mathcal E'$ contains rays that are exceptional for a criterion derived from (a) in the case $\dim_0 X> \dim\mathcal C$. 

We now state our results. The first one shows that the rays we call exceptional are rare. For surfaces in $\mathbb R^3$, it recovers the finiteness of the number of exceptional rays and the first item in Theorem~O'Shea--Wilson.

\begin{theorem}[Nowhere dense - Dimension]\label{Codim1}
Let $X\subset\mathbb R^n$ be a definable set of pure dimension $d>0$ at the origin $0\in \R^n$, $\mathcal C$ be its tangent cone at $0$, $\mathcal C'$ be the tangent cone to $(\overline{X})_{sing}$ at $0$, and $\mathcal E, \mathcal E'$ be given by~\eqref{E} and \eqref{E'} respectively. 
Then the following statements hold:
\begin{enumerate}[{\rm (i)}]
\item The set $\mathcal E'\cap\mathbb S^{n-1}$ is nowhere dense in $\mathcal C\cap\mathbb S^{n-1}$. In particular $\dim \mathcal E'<\dim \mathcal C\leqslant d.$
\item Let $\ell$ be a ray in $\mathcal C$ and $v\in\ell$. Assume that $\dim_v \mathcal C=d$. Then $\ell\subset \mathcal E\setminus \mathcal C_{sing}$ if and only if $\ell\subset \mathcal E'\setminus \mathcal C'.$ 
Furthermore, if $\ell\subset\mathcal C\setminus (\mathcal E\cup \mathcal C_{sing}\cup \mathcal C')$, then $\mathcal N_{\ell} = \{T_{v}\mathcal C\}$. In particular $\dim \mathcal E<d$. 
\end{enumerate}
\end{theorem}

\noindent  No analogue of the second item in Theorem~O'Shea--Wilson can be reached in full generality, according to Example~\ref{Codim2}. We however are able to  generalize it for hypersurfaces, as follows.

\begin{theorem}[Connected exceptional Nash fibers]\label{Connected}
Let $X\subset\mathbb R^n$, with $n\geqslant 2$, be a definable set of pure dimension $n-1$ at the origin $0\in \R^n$ and $\mathcal E$ be given by~\eqref{E}. Then for each ray $\ell\subset \mathcal E$, the Nash fiber $\mathcal N_\ell$ is a closed, connected and definable set of positive dimension.
\end{theorem}

\noindent It remains to get an analogue of the last item in Theorem~O'Shea--Wilson. 
For this, we study the rays $\ell$ for which the tangent cone to $\mathcal C$ along $\ell$ is not a plane of dimension $d$. In fact, the following result recovers the last item in Theorem~O'Shea--Wilson. 

\begin{theorem}[Singular cone]\label{Singular} Let $X\subset\mathbb R^n$ be a definable set of pure dimension $d>0$ at the origin $0\in \R^n$, $\mathcal C$ be its tangent cone at $0$, $\mathcal C'$ be the tangent cone to $(\overline{X})_{sing}$ at $0$, and $\mathcal E$ be given by~\eqref{E}. 
If $v\in \mathcal C\setminus\mathcal C'$ is a non zero point such that $C_v\mathcal C$ is not a plane of dimension $d$, then $v\in\mathcal E$.
\end{theorem}

The paper is organized as follows. We fix the notation which will be used throughout the paper in Section~\ref{OM}. 
This section also contains some basic properties of o-minimal structures needed in the paper. 
In Section~\ref{Definitions}, we give some elementary properties of tangent cones and tangent limits. The main results of the paper will be proved in Section~\ref{Proofs}. In the last section~\ref{Remarks}, we give some remarks and examples.

\section{Preliminaries} \label{OM}

\subsection{Notation}
For the remainder of the paper, we denote by $\|\cdot\|$ the Euclidean norm on $\R^n$ with respect to the Euclidean inner product $\langle\cdot,\cdot\rangle$. 
The closed ball, the open ball and the sphere centered at $x\in\R^n$ and of radius $r$ are denoted respectively by $\B^n_r(x), \ \mathring{\B}^n_r(x) $ and $\mathbb S^{n-1}_r(x)$. 
If $x=0$, we write $\B^n_r,\ \mathring{\B}^n_r$ and $\mathbb S^{n-1}_r$. 
If in addition $r=1$, then we write $\B^n,\ \mathring{\B}^n$ and $\mathbb S^{n-1}$. 
For $X\subset\R^n$, the sets $\overline X$ and $\partial X$ designate respectively the closure and the boundary of $X.$ The cardinality of $X$ is denoted by $\#(X)$.

Let $\dist(X,Y)$ stand for the Euclidean distance between two subsets $X$ and $Y$ of $\mathbb{R}^n,$ i.e.,
\begin{eqnarray*}
\dist(X, Y) := \inf\{\|x - y \|:\ x\in X,\ y \in Y \}.
\end{eqnarray*}
By convention, set $\dist(X,Y)=0$ if $X=\emptyset$ or $Y=\emptyset.$ 
Furthermore, the Hausdorff distance between $X$ and $Y$ is given by
$$\dist_{\mathcal H}(X, Y) := \max\Big\{\sup_{x\in X}\dist(x,Y),\sup_{y\in Y}\dist(y,X)\Big\}.$$

The set $Y\subset\mathbb R^n$ is called the limit of a sequence of subsets $Y_k$ of $\mathbb R^n$, i.e., $Y=\displaystyle\lim_{k\to +\infty}Y_k$, if and only if $\dist_{\mathcal H}(Y_k,Y)\to 0$ as $k\to+\infty$.

For a non empty definable set $X\subset\mathbb R^n$, let $X_{reg}$ be the set of regular points of $X$, which is the set of points where $X$ is a $C^1$-manifold. The complement of $X_{reg}$ in $X$ is denoted by $X_{sing},$ the set of singular points of $X$. 
Note that $X_{sing}$ is nowhere dense in $X.$ 

If $X$ is non empty, for a number $t \in \mathbb{R},$ let 
$$tX := \{tx : \ x\in X\}.$$

Let $v,w\in\mathbb R^n$ be not equal to $0$ simultaneously, denote by $\widehat{v,w}$ the angle between $v$ and $w.$ 
For convenience, if either $v=0$ or $w=0$, set $\displaystyle\widehat{v,w} := \frac{\pi}{2}.$ 
So $0\leqslant\widehat{v,w}=\widehat{w,v}\leqslant{\pi}.$ 
The angle between two rays $\ell_1$ and $\ell_2$, denoted by $\widehat{\ell_1,\ell_2}$, is defined to be the angle between the unit directions in each ray. 
If $V\ne\{0\}$ is a linear subspace of $\mathbb R^n,$ let $\pi_{V}$ be the orthogonal projection on $V$ and the angle between a non zero vector $v$ and $V$ is given by $$\angle(v,V)=\widehat{v,\pi_{V}(v)}.$$
For two linear subspaces $V_1\ne\{0\}$ and $V_2\ne\{0\}$ of $\mathbb R^n$, let $\pi_{V_i}$ be the orthogonal projection on $V_i$ ($i=1,2$). 
We define the angle between $V_1$ and $V_2$ by 
$$\begin{array}{lrll}
\angle(V_1,V_2)&:=&\left\{\begin{array}{llll}
\sup\{\widehat{v,\pi_{V_2}(v)}:\ v\in V_1\setminus\{0\}\} &\text{ if } \dim V_1\leqslant \dim V_2\\
\sup\{\widehat{v,\pi_{V_1}(v)}:\ v\in V_2\setminus\{0\}\} &\text{ if } \dim V_1\geqslant \dim V_2
\end{array}\right.\\
&=&\left\{\begin{array}{llll}
\max\{\widehat{v,\pi_{V_2}(v)}:\ v\in V_1\cap\mathbb S^{n-1}\} &\text{ if } \dim V_1\leqslant \dim V_2\\
\max\{\widehat{v,\pi_{V_1}(v)}:\ v\in V_2\cap\mathbb S^{n-1}\} &\text{ if } \dim V_1\geqslant \dim V_2.
\end{array}\right.
\end{array}$$
Observe that if $\dim V_1=\dim V_2,$ then
$$\sup\{\widehat{v,\pi_{V_2}(v)}:\ v\in V_1\setminus\{0\}\}=\sup\{\widehat {v,\pi_{V_1}(v)}:\ v\in V_2\setminus\{0\}\},$$
so the definition of angle between linear subspaces makes sense.
By definition, 
$$0\leqslant \angle(V_1,V_2)\leqslant \frac{\pi}{2}.$$ 
Furthermore, the equality $\angle(V_1,V_2)=0$ implies that $V_1\subseteq V_2$ or $V_2\subseteq V_1.$ If $V_1$ and $V_2$ are affine subspaces of $\mathbb R^n$, then the angle between $V_1$ and $V_2$ are determined by the angle between the corresponding parallel linear subspaces. It is not hard to verify that $\angle(\cdot,\cdot)$ defines a metric on the Grassmannian of $d$-dimensional linear subspaces of $\mathbb R^n$, for $1\leqslant d\leqslant n.$

\subsection{O-minimal structures}
The notion of o-minimality was developed in the late 1980s after it was noticed that many proofs of analytic and geometric properties of semi-algebraic sets and mappings can be carried over verbatim for sub-analytic sets and mappings.  We refer the reader to \cite{Coste2000, Dries1998, Dries1996} for the basic properties of o-minimal structures used in this paper. 

\begin{definition}{\rm 
A {\em structure} expanding the field of real numbers $\mathbb R$ is a collection 
$$\mathcal R = (\mathcal R_n)_{n\in\N},$$ 
where each $\mathcal R_n$ is a set of subsets of the affine space $\R^n$, satisfying the following axioms:
\begin{enumerate}[{\rm (a)}]
\item All algebraic subsets of $\R^n$ are in $\mathcal R_n$.
\item For every $n$, $\mathcal R_n$ is a Boolean subalgebra of the powerset of $\R^n$.
\item If $A\in \mathcal R_m$ and $B\in \mathcal R_n$, then $A \times B \in \mathcal R_{m+n}$.
\item If $p\colon\R^{n+1}\to\R^n$ is the projection on the first $n$ coordinates and $A\in\mathcal R^{n+1}$, then $p(A)\in\mathcal R_n$.\\
\noindent The elements of $\mathcal R_n$ are called the {\em definable subsets} of $\R^n$. A mapping whose graph is a definable set is called a {\em definable mapping}. The structure $\mathcal R$ is said to be {\em o-minimal} if, moreover, it satisfies the following axiom:
\item Each element of $\mathcal R_1$ is a finite union of points and intervals.
\end{enumerate}}
\end{definition}

Examples of o-minimal structures are:
\begin{itemize}
\item the semi-algebraic sets (by the Tarski--Seidenberg theorem),
\item the globally sub-analytic sets, i.e., the sub-analytic sets of $\mathbb{R}^n$ whose (compact) closures in $\mathbb{R}\mathbb{P}^n$ are sub-analytic (using Gabrielov's complement theorem).
\end{itemize}

From now on, we fix an arbitrary o-minimal structure expanding $\mathbb R$. The term ``definable'' means definable in this structure.
In the sequel, we will make use of the following Curve Selection lemma, Morse--Sard theorem and Hardt's definable triviality theorem. 
\begin{lemma}[Curve Selection]\cite[Lemma 3.1]{Milnor1968},~\cite[1.17]{Dries1996}\label{CurveSelectionLemma}
Let $X\subset \mathbb{R}^n$ be a definable set and $x\in \overline X\setminus X$. Then there is a $C^1$ definable curve $\gamma\colon(0,\varepsilon)\to X\setminus \{x\}$, for some $\varepsilon>0$, such that $\displaystyle\lim_{t\to 0^+}\gamma(t)=x.$
\end{lemma}

\begin{theorem}[Morse--Sard's Theorem]\cite[Theorem 1.4]{TLLoi2008},~\cite[Theorem 2.7]{Wilkie1999}\label{MorseSard}
Let $N$ and $M$ be $C^1$ definable manifolds of dimensions respectively $n$ and $m$ with $n\geqslant m\geqslant 1$, and $f\colon N\to M$ be a $C^1$ definable mapping. 
Let
$$\Sigma(f):= \{x \in N:\  \rank~ d_xf < m\}.$$
Then $f(\Sigma(f))$ is a definable set of dimension less than $m$.
\end{theorem}

\begin{theorem}[Hardt's triviality theorem]\cite[Theorem 5.22]{Coste2000},~\cite{Hardt1980},~\cite[Theorem 1.2, p. 142]{Dries1998}
\label{HardtTheorem}
Let $X$ and $Y$ be definable sets and $f \colon X \rightarrow Y$ be a continuous definable mapping.
Then there exists a finite partition $$Y=Y_1\sqcup\dots\sqcup Y_p$$ into definable subsets $Y_i,\ i = 1, \ldots p,$ such that $f$ is definably trivial over each $Y_i,$ i.e., $f^{-1}(Y_i)$ is definably homeomorphic to $f^{-1}(y_i)\times Y_i$ for each $i$ and any $y_i \in Y_i.$
\end{theorem}

Let $X\subset\mathbb R^n$  be a definable set. 
We define the {\em dimension} of $X$ by
$$\dim X:=\max\{\dim Y:\ Y \text{ is a } C^1 \text{-manifold contained in } X\}.$$
For $x\in \overline X$, the {\em dimension of $X$ at $x$} is defined by
$$\dim_x X:=\min\{\dim (X\cap U):\ U \text{ is an open neighborhood of }x\text{ in }\R^n\}.$$ 
Moreover, we say that $X$ is of {\em pure dimension $d$ at $x$} if there exists an open neighborhood $U$ of $x$ in $\R^n$ such that $\dim_yX=d$ for any $y\in X\cap U.$ 
Finally, we say that $X$ has {\em pure dimension} $d$ if $\dim_y X=d$ for any $y\in X$. 

\begin{lemma}\cite[1.16(3)]{Dries1996}\label{DimP} 
Let $X$ be a definable set in $\mathbb{R}^n.$ Then the following statements hold. 
\begin{enumerate}[{\rm (i)}]
\item If $X\ne\emptyset$ then $\dim(\overline{X}\setminus X)<\dim X.$ In particular, $\dim\overline{X}=\dim X.$
\item For any $x\in \overline X$, we have $\dim_x X=\dim_x\overline X.$ Moreover $X$ is of pure dimension $d$ at $x$ if and only if $\overline X$ is of pure dimension $d$ at $x.$
\end{enumerate}
\end{lemma}

\section{Tangent cones and tangent limits}\label{Definitions}

In this section, some elementary properties of tangent cones and tangent limits will be given. 
First of all, we state the following simple lemma whose proof is left to the reader.
\begin{lemma} \label{Lemma31}
Let $C \subset \mathbb{R}^n$ be a definable cone at the origin $0$ and let $D := C \cap \mathbb{S}^{n - 1}_R,$ where $R\in(0,+\infty).$ The following statements hold true:
\begin{enumerate}[{\rm (i)}]
\item $D$ is definable.

\item A non zero point $v\in C$ is a singular point of $C$ if and only if the ray through $v$ contains only singular points of $C.$ In particular, $C_{sing}$ is also a cone and we have 
$$C_{sing}\cap \mathbb S^{n-1}_R = D_{sing}.$$
\item If $v \in D \setminus D_{sing},$ then for all $t > 0$ we have $t v \in C \setminus C_{sing}$ and $T_{tv}C \cong  T_v D \oplus \mathbb{R}v.$
\item $C \setminus \{0\} $ is homeomorphic to $D \times (0, +\infty).$ In particular, $\dim_v C = \dim_vD + 1$ for all $v \in D$ and so $\dim C = \dim D + 1.$ 
\end{enumerate}
\end{lemma}

\begin{remark}{\rm 
In view of Lemma~\ref{Lemma31}(iv), the dimension of $ C$ at any point $v$ in a ray $\ell$ of $C$ is constant.  If $\ell$ is a non singular ray in $C$, for all $v\in\ell$, the tangent spaces $T_v C$ define the same plane in the Grassmannian $\mathbb G(\dim_v C,n)$. Moreover, it is not hard to check that the Nash fiber of $C$ along $\ell$ contains only one element which is $T_v C.$ 
}\end{remark}

For the remainder of the section, let $X$ be a definable set with $0\in \overline X$. Recall that $\mathcal C:=C_0 X$ is the tangent cone to $X$ at $0.$
The following lemma is a definable version of \cite[Lemma~1.2]{Kurdyka1989}.

\begin{lemma}\label{Lemma32} 
The set $\mathcal C$ is a nonempty closed definable cone of dimension at most $\dim_0 X.$
In addition, if $\dim_0 X>0$, then $\mathcal C$ has positive dimension.
\end{lemma}

\begin{proof} 
The last statement is clear so it remains to prove the first one.  By definition, it is easy to check that $\mathcal C$ is a nonempty closed definable cone. 
Let us show that $\dim \mathcal C  \leqslant \dim_0 X.$ To do this, for each $r > 0,$ define
$$A_r :=\Big\{\Big(\frac{x}{t}, t\Big) \in \mathbb{R}^n \times (0,+\infty): \ x \in X \cap \mathring{\mathbb{B}}^n_r\Big\}.$$
Obviously, $A_r$ is a nonempty definable set, which is homeomorphic to $(X \cap \mathring{\mathbb{B}}^n_r)\times (0,+\infty).$ Hence, for all $r$ sufficiently small, we have 
$$\dim A_r = \dim_0 X + 1.$$
Observe that 
$$\mathcal C \times \{0\} \subset \overline{A}_r \cap (\mathbb{R}^n \times \{0\}) \subset \overline{A}_r \setminus A_r.$$ 
Hence
$$\dim \mathcal C  \leqslant \dim (\overline{A}_r \setminus A_r) < \dim A_r = \dim_0 X + 1,$$
where the second inequality follows from Lemma~\ref{DimP}. 
This ends the proof of the lemma.
\end{proof}

Recall that the term ``ray'' means ``open ray emanating from the origin $0\in\R^n$''.
The following lemma shows that tangent cones and Nash fibers are invariant by taking closure.  
\begin{lemma}\label{NotEmpty} 
Suppose that $X$ is of pure dimension $d>0$ at the origin $0\in \R^n$. 
Then $C_0 X=C_0 \overline X$. 
Moreover, $\mathcal N_\ell(X)=\mathcal N_\ell(\overline X)\ne\emptyset$ for any ray $\ell\subset C_0 X,$ 
where $\mathcal N_\ell(X)$ and $\mathcal N_\ell(\overline X)$ are respectively the Nash fibers of $X$ and $\overline X$ along $\ell.$ 
\end{lemma}
\begin{proof} In view of Lemma~\ref{Lemma32}, we have $\dim C_0 X>0$. So $\dim C_0 \overline X>0$ as $C_0 X\subset C_0 \overline X$.
Let $\ell$ be a ray in $C_0\overline X$ and $x^k\in \overline X\setminus\{0\}$ be a sequence tending to $0$ such that $\displaystyle\frac{x^k}{\|x^k\|}\to v\in\ell$ as $k\to +\infty$. 
Since $\overline{X\setminus X_{sing}}=\overline X$ and since $X$ is of pure dimension $d>0$ at $0,$ for $k$ large enough, there is $y^k\in X\setminus X_{sing}$ such that 
\begin{equation}\label{xyk}\dim T_{y^k}X=d,\ \displaystyle\|y^k-x^k\|<\frac{\|x^k\|}{k},\ \text{ and }\ \angle(T_{y^k}X, T_{x^k}\overline X)<\frac{1}{k}\ \text{if}\ x^k\notin (\overline X)_{sing}.\end{equation}
Hence 
\begin{equation}\label{yk}\displaystyle\|y^k\|\leqslant\|x^k\|+\|y^k-x^k\|\leqslant\|x^k\|+\frac{\|x^k\|}{k}\to 0 \text { as } k\to +\infty.\end{equation}
On the other hand, 
$$\displaystyle\|y^k\|\geqslant\|x^k\|-\|y^k-x^k\|\geqslant\|x^k\|-\frac{\|x^k\|}{k}>\frac{\|x^k\|}{2}>0,$$
for $k$ large enough. 
Thus
$$\frac{\|y^k-x^k\|}{\|y^k\|}<\frac{\|x^k\|}{k\|y^k\|}<\frac{2}{k}\to 0\ \text{ as }\ k\to +\infty.$$
Therefore
$$\begin{array}{lll}
\displaystyle\left\|\frac{y^k}{\|y^k\|}-v\right\|&\leqslant&\displaystyle\left\|\frac{y^k-x^k}{\|y^k\|}\right\|+\left\|\frac{x^k}{\|y^k\|}-\frac{x^k}{\|x^k\|}\right\|+\left\|\frac{x^k}{\|x^k\|}-v\right\|\\
&<&\displaystyle\frac{2}{k}+\|x^k\|\left|\frac{1}{\|y^k\|}-\frac{1}{\|x^k\|}\right|+\left\|\frac{x^k}{\|x^k\|}-v\right\|\\
&=&\displaystyle\frac{2}{k}+\left|\frac{\|x^k\|-\|y^k\|}{\|y^k\|}\right|+\left\|\frac{x^k}{\|x^k\|}-v\right\|\\
&\leqslant&\displaystyle\frac{2}{k}+\frac{\|x^k-y^k\|}{\|y^k\|}+\left\|\frac{x^k}{\|x^k\|}-v\right\|<\frac{4}{k}+\left\|\frac{x^k}{\|x^k\|}-v\right\|\to 0 \text { as } k\to +\infty.
\end{array}$$
Combining this with~\eqref{yk} yields $v\in C_0 X$ and so $\ell\subset C_0 X.$
Hence $C_0 X=C_0 \overline X$.

Now taking a subsequence if necessary, we can suppose that the limit $\displaystyle\lim_{k\to +\infty}T_{y^k}X$ exists. 
Clearly, this limit belongs to $\mathcal N_\ell(X)$, i.e., $\mathcal N_\ell(X)\ne\emptyset$.

It remains to show that $\mathcal N_\ell(X)=\mathcal N_\ell(\overline X).$ 
For this, assume that 
$$x^k\notin (\overline X)_{sing},\ \dim_{x^k}\overline X=d\ \text{and the limit}\ P:=\lim_{k\to +\infty}T_{x^k}\overline X \ \text{exists},$$ 
so $P\in\mathcal N_\ell(\overline X).$ 
Then from~\eqref{xyk}, it follows that $P\in\mathcal N_\ell(X)$. 
Therefore $\mathcal N_\ell(\overline X)\subset\mathcal N_\ell(X)$ and so $\mathcal N_\ell(X)=\mathcal N_\ell(\overline X)$ as the inclusion $\mathcal N_\ell(X)\subset\mathcal N_\ell(\overline X)$ is clear. 
This ends the proof of the lemma.
\end{proof}

\begin{lemma}\label{Retraction} 
Suppose that $X$ is of pure dimension $d>0$ at the origin $0\in \R^n$.
Let $\ell$ be a ray in $\mathcal C$ and let $\ell_k\subset \mathcal C$ be a sequence of rays such that $\ell_k\to \ell$. Assume that, for each $k$, $P_k$ is a tangent limit to $X$ at $0$ along $\ell_k$ such that $P_k\to P$, then $P\in \mathcal N_\ell.$
\end{lemma}

\begin{proof} By definition and by the assumption, for each $k$, there is a sequence $x^{kl}\in X\setminus X_{sing}$ such that 
$$x^{kl}\to 0,\ \displaystyle\frac{x^{kl}}{\|x^{kl}\|}\to v_k\in\ell_k,\ \dim T_{x^{kl}}X=d\ \text{ and }\ T_{x^{kl}}X\to P_k\in\mathcal N_{\ell_k}\ \text{ as }\ l\to +\infty.$$ 
We can find a subsequence $x^{kl_k}$ such that
$$\displaystyle\|x^{kl_k}\|<\frac{1}{k},\ \left\|\frac{x^{kl_k}}{\|x^{kl_k}\|}- v_k\right\|<\frac{1}{k}\ \text{ and }\ \angle(T_{x^{kl_k}}X,P_k)<\frac{1}{k}.$$ 
Since $\ell_k \to \ell$, it follows that $v^k\to v,$ where $v$ is the unit direction in $\ell$, as $k\to+\infty.$
Then it is clear that 
$$x^{kl_k}\to 0,\ \displaystyle\frac{x^{kl_k}}{\|x^{kl_k}\|}\to v\ \text{and}\ T_{x^{kl_k}}X\to P\ \text{as}\ k\to+\infty.$$ 
Hence $P\in \mathcal N_\ell.$
\end{proof}

The following lemma shows that any tangent limit along a ray always contains that ray.

\begin{lemma}[{see also \cite[Lemma~4]{OShea2004},~\cite[Theorem 11.8, Theorem 22.1]{Whitney1965-2}}]\label{ContainRay} 
Suppose that $X$ is of pure dimension $d>0$ at the origin $0\in \R^n$.
Let $P$ be a tangent limit along a ray $\ell\subset\mathcal C$ and $v\in\ell$ be the unit direction in $\ell$. 
Then there exists a $C^1$ definable curve $\gamma\colon(0,\varepsilon)\to X\setminus X_{sing}$ such that 
$$\displaystyle\|\gamma(t)\|=t\ \text{for}\ t\in(0,\varepsilon),\ \displaystyle\lim_{t\to 0^+}\frac{\gamma(t)}{\|\gamma(t)\|}=v\ \text{ and }\ \displaystyle\lim_{t\to 0^+}T_{\gamma(t)}X=P.$$ 
Moreover $\ell\subset P$.
\end{lemma}
\begin{proof}
In view of Lemma~\ref{CurveSelectionLemma}, there is a $C^1$ definable curve $x\colon(0,\epsilon)\to X\setminus X_{sing}$ such that 
$$\displaystyle\lim_{r\to 0^+}x(r)=0,\ \displaystyle\lim_{r\to 0^+}\frac{x(r)}{\|x(r)\|}=v\ \text{ and }\ \displaystyle\lim_{r\to 0^+}T_{x(r)}X=P.$$ 
Let $t(r):=\|x(r)\|$ which is a definable function. 
By shrinking $\epsilon$, we may assume that $t(r)$ is of class $C^1$ and strictly increasing on $(0,\epsilon)$; moreover, the inverse function $r(t)$ is also of class $C^1$. 
Set $\gamma(t):=\displaystyle{x(r(t))}$. 
Then clearly 
$$\|\gamma(t)\|=\|x(r(t))\|=t(r(t))=t.$$
It is not hard to check that $\displaystyle\lim_{t\to 0^+}\frac{\gamma(t)}{\|\gamma(t)\|}=v$ and $\displaystyle\lim_{t\to 0^+}T_{\gamma(t)}X=P.$ Hence the first statement follows.

Now for all $t$, the point $\gamma(t) \in X$ so $\gamma'(t)\in T_{\gamma(t)}X$. 
Hence, when $t$ tends to $0^+$, L'Hospital's Rule gives
$$v = \lim_{t\to 0^+}\frac{\gamma(t)}{\|\gamma(t)\|} = \lim_{t\to 0^+}\frac{\gamma(t)}{t}=\lim_{t\to 0^+}\gamma'(t) \in \lim_{t\to 0^+} T_{\gamma(t)}X = P.$$ 
Therefore $\ell\subset P$. This ends the proof of the lemma. 
\end{proof}

\section{Proofs of the main results}\label{Proofs}

\begin{proof}[Proof of Theorem~\ref{Codim1}] 
(i) Observe that the second statement follows from the first one and Lemma~\ref{Lemma32}, so it remains to prove the first statement. Consider the definable function 
$$\rho\colon X\setminus X_{sing}\to[0,+\infty),\quad x\mapsto \|x\|.$$
Observe that the set of critical values of $\rho$ is finite in view of Theorem~\ref{MorseSard}. Thus there exists a constant $R > 0$ such that $\rho$ has no critical values in the interval $(0,R)$. 
Fix $\varepsilon\in(0,R)$ and let 
$$B := \{(u, r)\in\mathbb S^{n-1}\times (0,\varepsilon):\ r u\in X\setminus X_{sing}\}.$$ 
Clearly $B$ is definable and nonempty. Moreover, $B$ is non singular as it is the inverse image of the non singular set $(X\setminus X_{sing})\cap \mathring{\mathbb B}^n_{\varepsilon}$ by the diffeomorphism 
$$\phi \colon \mathbb S^{n-1}\times \mathbb (0,+\infty)\to\mathbb R^n\setminus\{0\},\quad (u, r)\mapsto r u.$$
Let $D:=(\mathcal C\setminus \mathcal C_{sing})\cap\mathbb S^{n-1}.$ 
Clearly $D$ is a non singular definable set. 
By Lemma~\ref{NotEmpty}, $\dim\mathcal C>0$, so $D$ is nonempty.
Set $\widetilde{B} := D\times \{0\}$ and consider the definable function 
$$f \colon \widetilde{B} \cup B \to \mathbb{R}, \quad (u, r)\mapsto r.$$ 
Note that $\widetilde{B} \subset \overline{B} \setminus B.$ Furthermore, we have
$$\mathrm{rank} (f|_{\widetilde{B}}) = 0 \quad \textrm{ and } \quad \mathrm{rank} (f|_{{B}}) = 1.$$
Indeed, the first equation is clear. To prove the second one, suppose for contradiction that $d_{(u, r)}f = 0$ for some $(u, r)\in B.$ 
We have $\|r u\|=r<\varepsilon < R$ and so $r u$ is not a critical point of $\rho.$ 
Moreover, 
$$T_{(u, r)}B = \ker d_{(u, r)}f \subset T_u \mathbb S^{n-1}\times \{0\}.$$
Therefore 
$$T_{r u}(X\setminus X_{sing})=d_{(u, r)}\phi (T_{(u, r)}B)\subset d_{(u, r)} \phi (T_u \mathbb S^{n-1}\times\{0\}) = T_{r u}\mathbb S^{n-1}_{r},$$ 
and so $r u$ is a critical point of $\rho,$ which is a contradiction. 

Let 
\begin{equation}\label {Z}Z:=\left\{\begin{array}{lll}v\in D:& \text{the pair } (B, \widetilde{B}) \text{ does not satisfy}\\ 
&\text{the Thom's } a_f \text{ condition at }(v,0)
\end{array}\right\}.\end{equation} 
Recall that the pair $(B, \widetilde{B})$ satisfies the {\em Thom's $a_f$ condition} at $(v, 0) \in \widetilde{B}$ if and only if for any sequence $w^k=(u^k,r_k)\in B$ converging to $(v, 0) \in \widetilde {B},$ we have 
$$\angle (T_{w^k}f^{-1}(r_k),T_v D\times \{0\})\to 0 \ \textrm{ as } \ k\to+\infty.$$
Let $D_1,\dots,D_q$ be the connected components of $D$.  
By~\cite[Lemma~2]{Ta1997} (see also \cite[Lemma~15]{Bekka1993} and \cite[Lemma 1.2]{Kurdyka1994}), we have $\dim (Z\cap D_i)<\dim D_i$ for each $i=1,\dots,q$. Hence $Z\cap D_i$ is nowhere dense in $D_i$, and so it is not hard to check that $Z$ is nowhere dense in $D$.

Now to prove (i), it is enough to show that $\mathcal E'\cap\mathbb S^{n-1}\subset Z,$ or, equivalently, 
$$D\setminus Z\subset D\setminus \mathcal E'.$$
To do this, take arbitrarily $v\in D\setminus Z$. We need to show that $v\not\in \mathcal E'.$ In fact, let $P$
be an arbitrary tangent limit to $X$ at $0$ along $\ell:=\mathbb R_{+}v.$ By definition, there is a sequence $x^k\in X\setminus X_{sing}$ such that 
$$x^k\to 0,\ \displaystyle\frac{x^k}{\|x^k\|}\to v\ \textrm{ and } \  T_{x^k}X\to P \ \textrm{ as } \ k \to +\infty.$$
Set $r_k:={\|x^k\|}$ and $u^k:=\displaystyle\frac{x^k}{\|x^k\|}$. 
Then the sequence $w^k := (u^k,r_k) \in B$ tends to $(v,0) \in \widetilde{B}.$ 
Since $v \not \in (Z\cup\mathcal C_{sing}),$ we have
$$\angle (T_{w^k}(f^{-1}(r_k)), T_v D\times\{0\}) \to 0 \ \textrm{ as } \ k \to +\infty.$$
Equivalently,
$$\angle\left(T_{w^k}\left[\left(\frac{1}{r_k}(X\setminus X_{sing})\cap\mathbb S^{n-1}\right)\times\{r_k\}\right], T_v D\times\{0\}\right) \to 0  \ \textrm{ as } \ k \to +\infty,$$
where $$\frac{1}{r_k}(X\setminus X_{sing})=\left\{\frac{x}{r_k}:\ x\in X\setminus X_{sing}\right\}.$$
Consequently,
\begin{equation}\label{PEqn2}
\angle\left(T_{u^k}\left[\frac{1}{r_k}(X\setminus X_{sing})\cap\mathbb S^{n-1}\right], T_v D \right) \to 0  \ \textrm{as } \ k \to +\infty.
\end{equation}

On the other hand, $\|x^k\| < R$ for all large $k.$ For all such $k,$ $x^k$ is a regular point of the function $\rho$ and so it is a non singular point of $X\cap\mathbb S^{n-1}_{{r_k}}.$ Combining this with the assumption that $X$ is of pure dimension $d$ at $0$, Lemma~\ref{Lemma31}(iv) and Lemma~\ref{Lemma32} yields
\begin{eqnarray*}\label{e-3}
\dim_{x^k}(X\cap\mathbb S^{n-1}_{{r_k}}) =  \dim_{x^k} X - 1=  \dim_0 X - 1 & \geqslant & \dim \mathcal C - 1 = \dim D \\
&\geqslant& \dim_v D=\dim T_vD.
\end{eqnarray*}
Then, by definition, we obtain
\begin{eqnarray*}
\angle(T_{x^k}X, T_v D)&=& \sup\{\widehat {z, \pi_{T_{x^k}X}(z)}: z\in T_vD\setminus\{0\}\}\\ 
&\leqslant & \sup\{\widehat {z,\pi_{T_{x^k}(X\cap\mathbb S^{n-1}_{{r_k}})}(z)}: z\in T_vD\setminus\{0\}\}\\ 
&= & \angle(T_{x^k}(X\cap\mathbb S^{n-1}_{{r_k}}), T_v D)\\
&= & \angle(T_{x^k}((X\setminus X_{sing})\cap\mathbb S^{n-1}_{{r_k}}), T_v D),
\end{eqnarray*}
where $\pi_{T_{x^k}X}$ and $\pi_{T_{x^k}(X\cap\mathbb S^{n-1}_{{r_k}})}$ are the orthogonal projections on ${T_{x^k}X}$ and $T_{x^k}(X\cap\mathbb S^{n-1}_{{r_k}})$ respectively. 
Observe that the homothety
$$(X\setminus X_{sing})\cap\mathbb S^{n-1}_{{r_k}} \to \frac{1}{r_k}(X\setminus X_{sing})\cap\mathbb S^{n-1}, \quad x \mapsto \frac{x}{r_k},$$ 
is a linear isomorphism. 
Thus, $T_{x^k}\big[(X\setminus X_{sing})\cap\mathbb S^{n-1}_{{r_k}}\big]$ and $T_{u^k}\big[\frac{1}{r_k}(X\setminus X_{sing})\cap\mathbb S^{n-1}\big]$ determine the same plane in the Grassmannian $\mathbb G(d-1,n)$. Therefore
\begin{eqnarray*}
\angle(T_{x^k}X, T_v D) & \leqslant & \angle \left (T_{u^k} \left[\frac{1}{r_k}(X\setminus X_{sing})\cap\mathbb S^{n-1}\right], T_v D \right).
\end{eqnarray*}
This, together with \eqref{PEqn2}, gives us $\angle(P,T_v D)=0.$ Note that $T_v \mathcal C =  T_v D\oplus \mathbb Rv$ (in light of Lemma~\ref{Lemma31}(iii)) and
$v \in P$ (by Lemma~\ref{ContainRay}). Hence, $\angle(P, T_v \mathcal C)=0.$ This yields $v\not\in \mathcal E'$ and so Item (i) follows. 

Before proving Item (ii), we need the following lemma.

\begin{lemma}\label{ContainPlane}
Let $X\subset\mathbb R^n$ be a definable set of pure dimension $d>0$ at the origin $0\in \R^n.$ 
For each ray $\ell\subset \mathcal C\setminus \mathcal C_{sing}$, there exists at least one tangent limit along $\ell$ containing $T_v\mathcal C$ for some $v\in\ell.$
\end{lemma}
\begin{proof}
In view of Lemma~\ref{NotEmpty}, we have $\mathcal N_\ell\ne\emptyset.$ 
From the definition of $\mathcal E'$, it is enough to prove the statement for any ray $\ell\subset \mathcal E'$. 

First of all, consider the case when $\ell$ is an isolated ray in $\mathcal C.$ 
In view of Lemma~\ref{ContainRay}, we have $\ell\subset P$ for any $P\in\mathcal N_\ell.$
Observe that $T_v \mathcal C$ is the line containing $\ell$. 
Hence $T_v \mathcal C\subset P$ for any $P\in\mathcal N_\ell.$

Now assume that $\ell$ is not an isolated ray in $\mathcal C$. 
As $\mathcal E'$ is nowhere dense in $\mathcal C$ by Item (i), it holds that $\mathcal C_{sing}\cup \mathcal E'$ is nowhere dense in $\mathcal C.$
Hence there exists a sequence of rays $\ell_k\subset \mathcal C\setminus (\mathcal C_{sing}\cup \mathcal E')$ converging to $\ell.$ For each $k$, let $v^k$ be a direction in $\ell_k$ such that the sequence $v^k$ tends to $v$ and let $P_k$ be a tangent limit of $X$ along $\ell_k$.  
Then $T_{v^k}\mathcal C\subseteq P_k$ by definition. 
Taking a subsequence if necessary, we may suppose that $P_k$ converges to $P.$ By Lemma~\ref{Retraction}, $P\in\mathcal N_\ell$. Since $T_{v^k}\mathcal C\to T_{v}\mathcal C$, we deduce that $T_{v}\mathcal C\subseteq P$. The lemma is proved.
\end{proof}

(ii) Let $\ell$ be a ray in $\mathcal C$ and let $v\in\ell$. Suppose that $\dim_v \mathcal C=d$.

Let us prove the first statement. Assume first that $\ell\subset \mathcal E\setminus \mathcal C_{sing}$. Then $\ell$ is not a ray in $\mathcal C'$ and by Lemma~\ref{ContainPlane}, there is $P\in\mathcal N_\ell$ containing $T_{v}\mathcal C$. 
By the assumption, $\dim_v \mathcal C=\dim P$. Hence $P= T_{v}\mathcal C.$ As $\#(\mathcal N_\ell)>1,$ there is $P'\in \mathcal N_\ell$ such that $P'\ne P.$ 
Hence $T_{v}\mathcal C\not\subset P',$ i.e., $\ell\subset\mathcal E'.$
Now suppose that  $\ell\subset \mathcal E'\setminus \mathcal C',$ so $\ell$ is not a ray in $\mathcal C_{sing}$ and there is $P\in\mathcal N_\ell$ not containing $T_{v}\mathcal C$. 
On the other hand, by Lemma~\ref{ContainPlane}, there is $P'\in\mathcal N_\ell$ such that $T_{v}\mathcal C\subset P'.$
Clearly $P\ne P'$, so $\ell\subset\mathcal E.$ The first statement follows. In fact, we have proved that
\begin{equation}\label{EE'}((\mathcal E\setminus \mathcal C_{sing})\cap \{u\in \mathcal C:\ \dim_u\mathcal C=d\})=((\mathcal E'\setminus \mathcal C')\cap \{u\in \mathcal C:\ \dim_u\mathcal C=d\})\end{equation}

Suppose that $\ell\subset \mathcal C\setminus (\mathcal E\cup\mathcal C_{sing}\cup \mathcal C')$. 
Then $\mathcal N_\ell$ contains only one element, say $P.$ In view of Lemma~\ref{ContainPlane},  we must have $T_{v}\mathcal C\subseteq P.$ 
As $\dim_v \mathcal C=\dim P$ by the assumption, we get $T_{v}\mathcal C= P,$ i.e., $\mathcal N_\ell=\{T_v\mathcal C\}$. 

Now it remains to show that $\dim \mathcal E< d.$ 
We have
$$\begin{array}{lll}
\mathcal E
&=&(\mathcal E\cap \{u\in \mathcal C:\ \dim_u\mathcal C=d\})\cup(\mathcal E\cap \{u\in \mathcal C:\ \dim_u\mathcal C<d\})\\
&=&((\mathcal E\setminus \mathcal C_{sing})\cap \{u\in \mathcal C:\ \dim_u\mathcal C=d\})\cup((\mathcal E\cap\mathcal C_{sing}\cap \{u\in \mathcal C:\ \dim_u\mathcal C=d\})\\
&&\cup (\mathcal E\cap \{u\in \mathcal C:\ \dim_u\mathcal C<d\})\\
&\subset&((\mathcal E\setminus \mathcal C_{sing})\cap \{u\in \mathcal C:\ \dim_u\mathcal C=d\})\cup \mathcal C_{sing}\cup (\mathcal E\cap \{u\in \mathcal C:\ \dim_u\mathcal C<d\})\\
&=&((\mathcal E'\setminus \mathcal C')\cap \{u\in \mathcal C:\ \dim_u\mathcal C=d\})\cup \mathcal C_{sing}\cup (\mathcal E\cap \{u\in \mathcal C:\ \dim_u\mathcal C<d\})\\
&\subset&\mathcal E'\cup \mathcal C_{sing}\cup (\mathcal E\cap \{u\in \mathcal C:\ \dim_u\mathcal C<d\}),
\end{array}$$
where the last equality follows from~\eqref{EE'}. 
Observe that $\dim \mathcal E'<d$ by Item (i) and the dimension of the third set is clearly smaller than $d.$ Moreover, in view of Lemma~\ref{Lemma32}, we get $\dim \mathcal C_{sing}<\dim \mathcal C\leqslant\dim_0 X=d.$ Hence $\dim \mathcal E<d.$ This ends the proof of the theorem.
\end{proof}
\begin{remark}{\rm
It is not hard to check that $\mathcal E'\cap \mathbb S^{n-1}=Z$ where $Z$ is defined by~\eqref{Z}.
}\end{remark}

By Lemma~\ref{Retraction} and Lemma~\ref{ContainPlane}, the following corollary is straightforward.
\begin{corollary}\label{TP}
Let $X\subset\mathbb R^n$ be a definable set of pure dimension $d>0$ at the origin $0\in \R^n.$ 
Let $\ell$ be a ray in $\mathcal C$ and let $\ell_k\subset \mathcal C\setminus \mathcal C_{sing}$ be a sequence of rays such that $\ell_k\to \ell$. Assume that $T_{v^k}\mathcal C\to \mathcal P$ for $v^k\in\ell_k$, then there exists $P\in \mathcal N_\ell$ such that $\mathcal P\subset P.$
\end{corollary}

Before proving Theorem~\ref{Connected}, we need some results of preparation. 
For each $x \in \mathbb{R}^n \setminus\{0\},$ let $\ell_x$ denote the open ray emanating from the origin through $x,$ namely, $\ell_x:=\{rx:\ r>0\}$. Given a real number $\delta \in [0,1]$ and a ray $\ell$, let $N_\delta(\ell)$ denote the {\em $\delta$-conical neighborhood} of $\ell,$ i.e.,   
\begin{equation*}\label{Nl}
N_\delta(\ell):=\{x \in \mathbb{R}^n\setminus\{0\}:\ \widehat{\ell,\ell_x}\leqslant\frac{\pi}{2}\ \text{and}\ \sin\widehat{\ell,\ell_x}\leqslant\delta\},
\end{equation*}
recall that $\widehat{\ell, \ell_x}$ denotes the angle between $\ell$ and $\ell_x.$ 

\begin{lemma}\label{Stable} Assume that $X$ is a definable set of positive pure dimension at $0$. Let $\ell$ be a ray in $\mathcal C$. Then there exists a constant $\delta_0>0$ and a continuous definable function $h\colon (0,\delta_0)\to\mathbb R$ with $h(\delta)\in (0,\delta]$ such that for $\delta\in(0,\delta_0)$, we have:
\begin{enumerate}[{\rm (i)}] 
\item The number of connected components of $(X\setminus\{0\})\cap N_\delta(\ell)\cap\mathring{\mathbb B}_{h(\delta)}^{n}$ is constant. 
\item For each connected component $Y$ of $(X\setminus\{0\})\cap N_\delta(\ell)\cap\mathring{\mathbb B}_{{h(\delta)}}^{n}$, the intersection $Y\cap N_{\delta'}(\ell)\cap\mathbb S_{r'}^{n-1}$ is nonempty and connected for $\delta'\in(0,\delta]$ and $r'\in(0, h(\delta'))$. 
\item If $\ell$ is not in $\mathcal C'$, then $((\overline X)_{sing}\setminus\{0\})\cap N_\delta(\ell)\cap\mathring{\mathbb B}_{h(\delta)}^{n}=\emptyset$.
\item The Nash fiber of $X$ along $\ell$ is equal to the Nash fiber of $(X\setminus\{0\})\cap N_\delta(\ell)\cap\mathring{\mathbb B}_{{h(\delta)}}^{n}$ along $\ell,$ i.e.,
$$\mathcal N_\ell(X)=\mathcal N_\ell((X\setminus\{0\})\cap N_\delta(\ell)\cap\mathring{\mathbb B}_{{h(\delta)}}^{n}).$$
\end{enumerate}
\end{lemma}

\begin{proof} Consider the set 
$$A:=\{(x,\delta,r)\in\R^n\times[0,1]\times(0,+\infty):\ x\in (X\setminus\{0\})\cap N_\delta(\ell)\cap\mathbb S_r^{n-1}\}$$ 
and the projection 
$$\proj\colon A\to[0,1]\times(0,+\infty),\ (x,\delta,r)\mapsto(\delta,r).$$
Let $v$ be an arbitrary point in $\ell$. 
Note that $A$ is a definable set since it can be written in the following form
$$A=\left\{\begin{array}{lll}
(x,\delta,r)\in\R^n\times[0,1]\times(0,+\infty):& x\in (X\setminus\{0\})\cap\mathbb S_r^{n-1},\\
&\ \dist(x,\ell)\leqslant\delta\|x\|,\\
&\langle x,v\rangle\geqslant 0
\end{array}\right\}$$ 
where $\dist(\cdot,\cdot)$ still denotes the Euclidean distance. In light of Theorem~\ref{HardtTheorem}, there is a partition of $[0,1]\times(0,+\infty)$ into disjoint definable sets $B_1,\dots,B_k$ such that $\proj$ is definably trivial over each $B_i$. Obviously, there is a unique set  $B\in\{B_1,\dots,B_k\}$ such that:
\begin{itemize}
\item [(a)] $\dim B=2$; and
\item [(b)] there exists a constant $\delta_0>0$ such that $[0,\delta_0]\times\{0\}\subset \overline B.$ 
\end{itemize}
For $\delta\in(0,\delta_0),$ if there is $r>0$ such that $\{\delta\}\times(0,r)\cap B=\emptyset$, then set $h(\delta)=0$; otherwise, set
$$h(\delta):=\sup\{r\in (0,\delta):\ \{\delta\}\times(0,r)\subset B\}.$$
Clearly $h$ is a definable function.
Shrinking $\delta_0$, if necessary, so that either $h(\delta)> 0$ for all $\delta\in(0,\delta_0)$ or $h\equiv 0$ on $(0,\delta_0).$
Observe that the latter can not happen since otherwise, we shall have $[0,\delta_0]\times\{0\}\not\subset \overline B,$ which is a contradiction.

Now by triviality, Items (i) and (ii) follows. 

As $\mathcal C'$ is closed and $\ell$ is not a ray in $\mathcal C'$, by shrinking $\delta_0$, we have $N_{\delta_0}(\ell)\cap \mathcal C'=\{0\}$. Hence, it is not hard to see that there is a constant $r_0>0$ such that $$((\overline X)_{sing}\setminus\{0\})\cap N_{\delta_0}(\ell)\cap \mathring{\mathbb B}^n_{r_0}=\emptyset.$$
By shrinking $\delta_0$, if necessary, so that $\delta_0<r_0$, Item (iii) follows.

Let $P\in\mathcal N_\ell(X)$. By Lemma~\ref{ContainRay}, there exists a $C^1$ definable curve 
$\gamma\colon(0,\varepsilon)\to X\setminus X_{sing},$ such that 
$$\displaystyle\|\gamma(t)\|=t\ \text{for}\ t\in(0,\varepsilon),\ \displaystyle\lim_{t\to 0^+}\frac{\gamma(t)}{\|\gamma(t)\|}=v\in\ell\ \text{ and }\ \displaystyle\lim_{t\to 0^+}T_{\gamma(t)}X=P.$$ 
Shrinking $\varepsilon$ so that $\varepsilon\in(0,h(\delta))$ and $\displaystyle\frac{\gamma(t)}{\|\gamma(t)\|}\in N_{\delta}(\ell)$. 
Clearly $\gamma(t)\in (X\setminus\{0\})\cap N_\delta(\ell)\cap\mathring{\mathbb B}_{h(\delta)}^{n}$ for all $t\in(0,\varepsilon).$
Therefore $P\in \mathcal N_\ell((X\setminus\{0\})\cap N_\delta(\ell)\cap\mathring{\mathbb B}_{h(\delta)}^{n}).$
Consequently 
$$\mathcal N_\ell(X)\subset\mathcal N_\ell((X\setminus\{0\})\cap N_\delta(\ell)\cap\mathring{\mathbb B}_{h(\delta)}^{n}).$$
Now Item (iv) follows from observing that the inclusion $\mathcal N_\ell((X\setminus\{0\})\cap N_\delta(\ell)\cap\mathring{\mathbb B}_{h(\delta)}^{n})\subset \mathcal N_\ell(X)$ is trivial.
\end{proof}

\begin{lemma}\label{RegularConnected} 
With the notation in Lemma~\ref{Stable}, let $\ell$ be a ray in $\mathcal C\setminus \mathcal C'$ and let $Y$ be a connected component of $(X\setminus\{0\})\cap N_\delta(\ell)\cap\mathring{\mathbb B}_{h(\delta)}^{n}$, where $\delta\in(0,\delta_0)$. Then $\mathcal N_\ell(Y)$ is connected. 
\end{lemma}

\begin{proof} 
Without loss of generality, we may suppose that $(X\setminus\{0\})\cap N_\delta(\ell)\cap\mathring{\mathbb B}_{h(\delta)}^{n}$ is connected, so 
\begin{equation*}\label{Yconnected}Y=(X\setminus\{0\})\cap N_\delta(\ell)\cap\mathring{\mathbb B}_{h(\delta)}^{n}.\end{equation*}
Furthermore, by Lemma~\ref{Stable}(iv), $\mathcal N_\ell(Y)= \mathcal N_\ell(X)\subset\mathbb G(n-1,n)$. 
Assume for contradiction that $\mathcal N_\ell(Y)$ is not connected. Let $M_1$ and $M_2$ be two distinct connected components of $\mathcal N_\ell(X)$ which are clearly closed sets in $\mathbb G(n-1,n)$. Let $\eta>0$ be such that 
\begin{equation}\label{UM1}U_\eta(M_1)\cap (\mathcal N_\ell\setminus M_1)=\emptyset, \end{equation}
where $U_\eta(M_1)$ is the closed neighborhood of radius $\eta$ of $M_1$, namely,
$$U_\eta(M_1)=\{P\in\mathbb G(n-1,n):\ \text{there is } Q\in M_1 \text{ such that } \angle(P,Q)\leqslant\eta\}.$$ 
Let $P_1\in M_1$ and $P_2\in M_2$. By Lemma~\ref{ContainRay}, there exist two $C^1$ definable curves 
$$\gamma_1,\gamma_2\colon (0,\varepsilon)\to Y$$ 
such that:
\begin{itemize} 
\item [(a)] $\|\gamma_1(t)\|=\|\gamma_2(t)\|=t$ for $t\in(0,\varepsilon)$;
\item [(b)] $\gamma_1(t)\cap  X_{sing}=\emptyset$ and $\gamma_2(t)\cap  X_{sing}=\emptyset$ for $t\in(0,\varepsilon)$;
\item [(c)] $\displaystyle\lim_{t\to 0^+}\frac{\gamma_1(t)}{\|\gamma_1(t)\|}=\lim_{t\to 0^+}\frac{\gamma_2(t)}{\|\gamma_2(t)\|}=v\in\ell$; 
\item [(d)] $\displaystyle\lim_{t\to 0^+}T_{\gamma_1(t)}X=P_1$ and $\displaystyle\lim_{t\to 0^+}T_{\gamma_2(t)}X=P_2$.
\end{itemize}
Shrinking $\varepsilon$ so that $\varepsilon\leqslant h(\delta)$ and let 
$$g(t):=\max\left\{\sup_{0<t'\leqslant t}\sin\widehat{\ell,\ell_{\gamma_1(t')}},\sup_{0<t'\leqslant t}\sin\widehat{\ell,\ell_{\gamma_2(t')}}\right\}.$$
Clearly 
\begin{equation}\label{rdelta}g(t)\to 0 \text{ as } t\to 0^+\ \text{ and }\ g(t)\leqslant\delta \text{ for } t\in(0,h(\delta)).\end{equation}
The latter, together with the assumption that $Y=(X\setminus\{0\})\cap  N_\delta(\ell) \cap \mathring{\mathbb B}_{h(\delta)}^{n}$ is connected and Lemma~\ref{Stable}(ii), implies that, for $t\in(0,\varepsilon)$, the set $(X\setminus\{0\})\cap N_{g(t)}(\ell)\cap\mathbb S_{b(t)}^{n-1}$ is also connected, where 
\begin{equation}\label{bt}
b(t)=\frac{\min\{h(g(t)),h(\delta)\}}{2}.
\end{equation} 
Hence there exists a continuous curve 
$$\alpha_t\colon[0,1]\to (X\setminus\{0\})\cap N_{g(t)}(\ell)\cap\mathbb S_{b(t)}^{n-1}, \ s\mapsto\alpha_t(s),$$ 
such that $\alpha_t(0)=\gamma_1(b(t))$ and $\alpha_t(1)=\gamma_2(b(t))$. 
Obviously, the mapping $$s\mapsto \beta_t(s):=T_{\alpha_t(s)}X$$ is continuous. 
By Item (d), for $t>0$ small enough, we have 
$$\beta_t(0)\in U_\eta(M_1)\ \text{ and }\ \beta_t(1)\not\in U_\eta(M_1).$$ 
In particular, these hold for $t=\frac{1}{k}$ with $k\in\mathbb N$ large enough. 
By continuity, for all such $k$, there is $s_k\in(0,1)$ such that $\beta_{\frac{1}{k}}(s_k)\in \partial U_\eta(M_1)$. 
By the compactness of $\partial U_\eta(M_1)$, which follows from the closedness of $\partial U_\eta(M_1)$ and the compactness of $\mathbb G(n-1,n)$, the sequence $T_{\alpha_{\frac{1}{k}}(s_k)}X=\beta_{\frac{1}{k}}(s_k)$ has an accumulation point in $\partial U_\eta(M_1)$, say $P.$ 
From the definition of $g(t),\ \alpha_t$,~\eqref{rdelta} and~\eqref{bt}, we have
$$\sin\widehat{\ell,\ell_{\alpha_{\frac{1}{k}}(s_k)}}\leqslant g\left(\frac{1}{k}\right)\to 0 \text{ and } \|\alpha_{\frac{1}{k}}(s_k)\|=b\left(\frac{1}{k}\right)\leqslant\frac{h(g(\frac{1}{k}))}{2}\leqslant\frac{g(\frac{1}{k})}{2}\to 0 \text{ as } k\to+\infty .$$
Thus, $P\in \mathcal N_\ell$. 
Therefore $$\emptyset\ne\partial U_\eta(M_1)\cap\mathcal N_\ell=\partial U_\eta(M_1)\cap(\mathcal N_\ell\setminus M_1)\subset U_\eta(M_1)\cap(\mathcal N_\ell\setminus M_1).$$ 
This contradicts~\eqref{UM1} and so ends the proof of the lemma.
\end{proof}

We also need the following key lemma which relates the differential of the distance function between two disjoint non singular hypersurfaces with the angle between the corresponding tangent hyperplanes.
\begin{lemma}\label{Key} Let $Y,Z\subset\mathbb R^n$, with $n\geqslant 2$, be two disjoint non singular hypersurfaces. Define the function 
$$\widetilde\rho:\mathbb R^n\times \mathbb R^n\to\mathbb R,\ (y,z)\mapsto \|y-z\|=\sqrt{\sum_{i=1}^n(y_i-z_i)^2}.$$ For any $y\in Y$ and $z\in Z$, the following statements hold:
\begin{itemize}
\item [(i)] If $(y,z)$ is a critical point of the restriction $\widetilde\rho|_{Y\times Z}$, then 
$\angle(T_yY,T_zZ)=0.$
\item [(ii)] For each $(y,x)\in Y\times Z,$  we have
\begin{equation*}\label{BoundBelow}\|\nabla(\widetilde\rho|_{Y\times Z})(y,z)\|\geqslant \frac{\sin\angle(T_y Y,T_z Z)}{\sqrt{2}}, \end{equation*} 
where $\nabla(\widetilde\rho|_{Y\times Z})$ denotes the gradient of $\widetilde\rho|_{Y\times Z}.$
\end{itemize}
\end{lemma}
\begin{proof} (i) Assume that $(y, z)$ is a critical point of $\widetilde\rho|_{Y\times Z},$ i.e.,  
$$\displaystyle d_{(y,z)}\widetilde\rho(T_yY\times T_zZ) = 0.$$ 
Hence, for any $(a,0)\in T_y Y\times \{0\}\subset T_yY\times T_zZ$, we have 
\begin{eqnarray*}
0= d_{(y,z)} \widetilde\rho(a,0 )  =\left\langle\frac{(y-z,z-y)}{\|y-z\|},(a,0)\right\rangle= \frac{\langle y - z, a \rangle}{\|y - z\|}.
\end{eqnarray*}
Consequently, $y - z \perp T_yY.$ Similarly, we have also $y-z\perp T_zZ$. Therefore, Item (i) follows from remarking that $\dim T_yY=\dim T_zZ=n-1$ by assumption. 

(ii) Since the statement is clear if $\angle(T_yY,T_zZ)=0$, let us suppose that $\angle(T_y Y,T_z Z)\ne 0$. 
Set $$W=(\{y\}+T_yY)\cap (\{z\}+T_zZ).$$ 
Note that $T_yY$ and $T_zZ$ are of dimension $n-1$, so $W\ne\emptyset$. We consider two cases:

\subsubsection*{Case 1: $y,z\in W$.} We have 
$$\nabla\widetilde\rho(y,z)=\frac{(y-z,z-y)}{\|y-z\|}\in T_yY\times T_zZ.$$ 
This, together with the fact that $\nabla(\widetilde\rho|_{Y\times Z})(y,z)$ is the orthogonal  projection of $\nabla\widetilde\rho(y,z)$ on $T_yY\times T_zZ$, implies
$$\displaystyle\|\nabla(\widetilde\rho|_{Y\times Z})(y,z)\|=\|\nabla\widetilde\rho(y,z)\|=\left\|\frac{(y-z,z-y)}{\|y-z\|}\right\|=\sqrt{2}>\frac{\sin\angle(T_y Y,T_z Z)}{\sqrt{2}}.$$

\subsubsection*{Case 2: $y\not\in W$ or $z\not\in W$.} Assume with no loss of generality that $y\not\in W$. 
Let $w$ be the orthogonal projection of $y$ on $W$. 
We denote by $L_y$ the line through $w$ and $y$, and by $L_z$ the line through $w$ and $z$ if $w\ne z$. If $w=z$, we let $L_z$ be any line in $W$ through $z$. 
It is clear that $$\angle(L_y,L_z)\geqslant\angle(T_yY,T_zZ).$$  
Furthermore, since $\widetilde L_y\times \widetilde L_z$ is a subspace of  ${T_yY\times T_zZ}$, where 
$$\widetilde L_y:=\{-y\}+ L_y\ \text{ and }\ \widetilde L_z:=\{-z\}+ L_y,$$ 
we get
$$\|\nabla(\widetilde\rho|_{Y\times Z})(y,z)\|=\|\proj_{T_yY\times T_zZ}\nabla\widetilde\rho(y,z)\|\geqslant\|\proj_{\widetilde L_y\times \widetilde L_z}\nabla\widetilde\rho(y,z)\|,$$ 
where $\proj_{T_yY\times T_zZ}$ and $\proj_{\widetilde L_y\times \widetilde L_z}$ are respectively the orthogonal projections on $T_yY\times T_zZ$ and $\widetilde L_y\times \widetilde L_z.$ 
So to prove the statement in this case, it is enough to show that 
\begin{equation}\label{F1}
\|\proj_{\widetilde L_y\times \widetilde L_z} \nabla\widetilde\rho(y,z)\|\geqslant\frac{\sin\angle(L_y,L_z)}{\sqrt{2}}.
\end{equation} 
Let $y'$ and $z'$ be the orthogonal projection of $y$ and $z$ on $L_z$ and $L_y$, respectively. Then it is clear that 
\begin{equation}\label{F2}
\begin{array}{lll}
\|\proj_{\widetilde L_y\times \widetilde L_z}\nabla\widetilde\rho(y,z)\|
&=&\displaystyle\left \|\proj_{\widetilde L_y\times \widetilde L_z}\frac{(y-z,z-y)}{\|y-z\|}\right \|\\
&=&\displaystyle\frac{\big\|(\displaystyle\proj_{\widetilde L_y}(y-z),\proj_{\widetilde L_z}(z-y))\big\|}{\|y-z\|} \\
&=&\displaystyle\frac{\sqrt{\|y-z'\|^2+\|z-y'\|^2}}{\|y-z\|} \text{ (see Figure 1)}.
\end{array}
\end{equation}
If $z=z'$, then we have
$$\|\proj_{\widetilde L_y\times \widetilde L_z}\nabla\widetilde\rho(y,z)\|\geqslant\frac{\|y-z'\|}{\|y-z\|}=\frac{\|y-z\|}{\|y-z\|}=1,$$
and~\eqref{F1} follows. So suppose that $z\ne z'$. Remark that $y\ne y'$ as $y\not\in W$ and $y'\in W$. Denote by $x$ the intersection point of the line through $y,y'$ and the line through $z,z'$. Then it is not hard to check that (see Figure 1)
$$\|y-z'\|=\|y-x\|\sin\angle(L_y,L_z) \ \text{ and }\ \ \|z-y'\|=\|z-x\|\sin\angle(L_y,L_z).$$ 
This and~\eqref{F2} together yield
\begin{eqnarray*}
\|\proj_{\widetilde L_y\times \widetilde L_z}\nabla\widetilde\rho(y,z)\| & = & \frac{\sqrt{\|y-x\|^2\sin^2\angle(L_y,L_z)+\|z-x\|^2\sin^2\angle(L_y,L_z)}}{\|y-z\|}\\
&= &\frac{\sqrt{\|y-x\|^2+\|z-x\|^2}\sin\angle(L_y,L_z)}{\|y-z\|}\\
&\geqslant &\frac{\sqrt{(\|y-x\|+\|z-x\|)^2}\sin\angle(L_y,L_z)}{\sqrt{2}\|y-z\|}\\
&\geqslant &\frac{\|y-z\|\sin\angle(L_y,L_z)}{\sqrt{2}\|y-z\|}=\frac{\sin\angle(L_y,L_z)}{\sqrt{2}},
\end{eqnarray*}
which proves \eqref{F1} and so the lemma follows.
\end{proof}

\begin{tikzpicture}[scale=1]
\def\d{3};\def\k{0.5};\def\r{3.5};\def\m{2};\def\s{0.2}; \def\c{1.5};\def\e{2.5}; \def\f{0.75}; \def\g{1.25}

\coordinate (u) at ($\r/sqrt(2)*(1,1)$);
\coordinate (v) at ($\r/sqrt(2)*(-1,1)$);
\coordinate (w) at ($\r/sqrt(2)*(-1,0)$);
\coordinate (z) at ($\r/sqrt(2)*(0,1)$);

\coordinate (O1) at (-\d-\k,0);

\draw [thick](-2*\d-\k,0) -- (-\k,0);
\draw [thick](-2*\d-\k,{-\f*\r/sqrt(2)}) -- (0,{-\f*\r/sqrt(2)});
\draw [thick]($(O1)+(v)+(0,6*\k/\m)$) -- ($(O1)-(u)$);
\draw [domain=-2*\d-0.5*\k: -1.5*\k,thick,variable=\t] plot ({\t},{-\t-\d-\k});
\draw [domain=-2*\d-0.5*\k: 0,thick,variable=\t] plot ({\t},{-\t-\d-\k+\k*\r/sqrt(2)});
\draw [domain=-2*\d-0.5*\k: -\d,thick,variable=\t] plot ({\t},{\t+\d+\k+\k*\r/sqrt(2)});
\draw [thick]($(O1)+(w)+\s/\e*(z)$) -- ($(O1)+(w)+\s/\e*(z)-\s/\e*(w)$);
\draw [thick]($(O1)+(w)-\s/\e*(w)$) -- ($(O1)+(w)+\s/\e*(z)-\s/\e*(w)$);
\draw [thick]($(O1)+\k*\k*(v)+\s/\r*(w)-\s/\r*(z)$) -- ($(O1)+\k*\k*(v)+2*\s/\r*(w)$);
\draw [thick]($(O1)+\k*\k*(v)+\s/\r*(w)+\s/\r*(z)$) -- ($(O1)+\k*\k*(v)+2*\s/\r*(w)$);

\node [right] at ($(O1)+(0,\k/\m)$) {$w$};
\node [right] at (-2*\k+\k,0) {$L_z$};
\node [below] at (-2*\k,-\d+0.5*\k) {$L_y$};
\node [right] at ($(O1)+(v)+(0,\k/\m)$) {$y$};
\node [left] at ($(O1)+(w)+(0,\k/\m)$) {$y'$};
\node [left] at ($(O1)+\k*(w)+(0,\k/\m)$) {$z$};
\node [above] at ($(O1)+\k*\k*(v)+(0,\k/\m/\m)$) {$z'$};
\node [right] at ($(O1)+(w)-\k*(z)-(0,\k/\m)$) {$x$};
\node [right] at ($(O1)-1.25*(w)-0.75*(z)+(0,\k/\m)$) {$0$};
\node [left] at ($(O1)+1.25*(w)-0.75*(z)$) {$\widetilde L_z$};
\node [below] at ($(O1)-1.25*(w)-0.75*(z)+(\k/\m,-\k/\m)$) {$\widetilde L_y$};

\filldraw (-\d-\k,0) circle (2pt);
\filldraw ($(O1)+(v)$) circle (2pt);
\filldraw ($(O1)+(w)$) circle (2pt);
\filldraw ($(O1)+\k*(w)$) circle (2pt);
\filldraw ($(O1)+\k*\k*(v)$) circle (2pt);
\filldraw ($(O1)+(w)-\k*(z)$) circle (2pt);
\filldraw ($(O1)-1.25*(w)-0.75*(z)$) circle (2pt);

\draw [thick,domain=135:180] plot ({-\d-\k+\c*\s*cos(\x)}, {\c*\s*sin(\x)});
\draw [thick,domain=45:90] plot ({-\d-\k-\r/sqrt(2)+\c*\s*cos(\x)}, {-\k*\r/sqrt(2)+\c*\s*sin(\x)});
\draw [thick,domain=135:180] plot ({-\d-\k+\g*\r/sqrt(2)+\c*\s*cos(\x)}, {-\f*\r/sqrt(2)+\c*\s*sin(\x)});


\coordinate (O2) at (\d+\k,0);

\draw [thick](\k,0) -- (2*\d+\k,0);
\draw [thick](\k,{-\f*\r/sqrt(2)}) -- (2*\d+2.5*\k,{-\f*\r/sqrt(2)});
\draw [thick]($(O2)+(v)+(0,6*\k/\m)$) -- ($(O2)-(u)-(0,7*\k)$);
\draw [domain=2*\k-0.5*\k: 2*\d+0.5*\k,thick,variable=\t] plot ({\t},{-\t+\d+\k});
\draw [domain=2*\k-0.5*\k: 2*\d+2*\k,thick,variable=\t] plot ({\t},{-\t+\d+\k+\k*\r/sqrt(2)});
\draw [domain=2*\k-0.5*\k: 2*\d+0.5*\k,thick,variable=\t] plot ({\t},{\t-\d-\k-\r/sqrt(2)});
\draw [thick]($(O2)+(w)+\s/\e*(z)$) -- ($(O2)+(w)+\s/\e*(z)-\s/\e*(w)$);
\draw [thick]($(O2)+(w)-\s/\e*(w)$) -- ($(O2)+(w)+\s/\e*(z)-\s/\e*(w)$);
\draw [thick]($(O2)-\k*(v)+\s/\r*(w)-\s/\r*(z)$) -- ($(O2)-\k*(v)+2*\s/\r*(w)$);
\draw [thick]($(O2)-\k*(v)+\s/\r*(w)+\s/\r*(z)$) -- ($(O2)-\k*(v)+2*\s/\r*(w)$);

\node [right] at ($(O2)+(0,\k/\m)$) {$w$};
\node [right] at (2*\d+\k,0) {$L_z$};
\node [below] at (2*\d,-\d+0.5*\k) {$L_y$};
\node [right] at ($(O2)+(v)+(0,\k/\m)$) {$y$};
\node [left] at ($(O2)+(w)+(0,\k/\m)$) {$y'$};
\node [left] at ($(O2)-(w)+(0,\k/\m)$) {$z$};
\node [above] at ($(O2)-\k*(v)+(0,\k/\m/\m)$) {$z'$};
\node [right] at ($(O2)+(w)-2*(z)-(0,\k/\m)$) {$x$};
\node [right] at ($(O2)-1.25*(w)-0.75*(z)+(0,\k/\m)$) {$0$};
\node [right] at ($(O2)-1.5*(w)-0.75*(z)$) {$\widetilde L_z$};
\node [below] at ($(O2)-1.25*(w)-0.75*(z)+(\k/\m,-\k/\m)$) {$\widetilde L_y$};

\filldraw (\d+\k,0) circle (2pt);
\filldraw ($(O2)+(v)$) circle (2pt);
\filldraw ($(O2)+(w)$) circle (2pt);
\filldraw ($(O2)-(w)$) circle (2pt);
\filldraw ($(O2)-\k*(v)$) circle (2pt);
\filldraw ($(O2)+(w)-2*(z)$) circle (2pt);
\filldraw ($(O2)-1.25*(w)-0.75*(z)$) circle (2pt);

\draw [thick,domain=135:180] plot ({\d+\k+\c*\s*cos(\x)}, {\c*\s*sin(\x)});
\draw [thick,domain=45:90] plot ({\d+\k-\r/sqrt(2)+\c*\s*cos(\x)}, {-2*\r/sqrt(2)+\c*\s*sin(\x)});
\draw [thick,domain=135:180] plot ({\d+\k+\g*\r/sqrt(2)+\c*\s*cos(\x)}, {-\f*\r/sqrt(2)+\c*\s*sin(\x)});

\end{tikzpicture}
$$\textrm{Figure 1.}$$

Now we are in position to prove Theorem~\ref{Connected}.
\begin{proof}[Proof of Theorem~\ref{Connected}]
In view of Lemma~\ref{NotEmpty}, we may assume that $X$ is closed. 
Let $\ell$ be any ray in $\mathcal E$ with the unit direction $v$. It is not hard to show that $\mathcal N_\ell$ is a closed definable set, so we leave it to the reader. 
Note that if $\mathcal N_\ell$ is connected, then it is path connected and so is of positive dimension as $\#(\mathcal N_\ell)>1$, so it remains to prove the connectedness. 
	Let $X_i=X_i(\delta)$  with $i=1,\dots,m,\ \delta\in(0,\delta_0)$, be the connected components of $(X\setminus\{0\})\cap N_\delta(\ell)\cap\mathring{\mathbb B}_{h(\delta)}^{n}$, where $\delta_0$ and $h(\delta)$ are the constants in Lemma~\ref{Stable}. 
Then it is not hard to see that $\ell\subset C_0(X_i)$ for any $i$. Let $\mathcal N_\ell(X_i)$ be the Nash fiber of $X_i$ over $0$ along $\ell$. 
Obviously $$\mathcal N_\ell=\bigcup_{i=1}^m\mathcal N_\ell(X_i)$$
in light of Lemma~\ref{Stable}(iv). 
Furthermore, $\mathcal N_\ell(X_i)$ is connected in view of Lemma~\ref{RegularConnected}. Hence, to prove that $\mathcal N_\ell$ is connected, it is enough to show that 
$$\mathcal N_\ell(X_i)\cap \mathcal N_\ell(X_j)\ne\emptyset\ \text{ for }\ i\ne j.$$ 

In order to do this, let $Y,Z\in\{X_i,\ i=1,\dots,m\}$ with $Y\ne Z$. 
Shrinking $\delta$ if necessary so that $\delta<1.$
By the construction, for $0<r<\frac{h(\delta)}{2}$, we have $\mathring{\mathbb B}_{r\delta}^{n}(rv)\subset N_\delta(\ell)\cap \mathring{\mathbb B}_{h(\delta)}^{n}\setminus\{0\}$, recall that $\mathring{\mathbb B}_{r\delta}^{n}(rv)$ is the open ball of radius $r\delta$ centered at $rv$. 
Set 
$$Y_r=Y\cap \mathring{\mathbb B}_{r\delta}^{n}(rv)\ \text{ and }\ Z_r=Z\cap \mathring{\mathbb B}_{r\delta}^{n}(rv).$$ 
Clearly $Y_r$ and $Z_r$ are non singular hypersurfaces in view of Lemma~\ref{Stable}(iii). 
Hence we can define the following function $\theta$ given by
$$\theta(r):=\inf_{y\in Y_r,z\in Z_r} \sin\angle(T_y Y_r,T_z Z_r)=\max\{\|u-\pi_{T_y Y_r}(u)\|:\ u\in T_z Z_r\cap\mathbb S^{n-1}\},$$
recall that $\pi_{T_y Y_r}$ is the orthogonal projection on ${T_y Y_r}.$
Then $\theta$ is a non negative definable function. Assume that we have proved
\begin{eqnarray} \label{PT5}
\lim_{r \to 0^+} \theta(r) = 0.
\end{eqnarray}
This, of course, implies that $\mathcal N_\ell(Y)\cap \mathcal N_\ell(Z)\ne\emptyset$ by shrinking $\delta$ to 0. So it remains to prove~\eqref{PT5}. For contradiction, suppose that $\theta(r)\not\to 0$. Then there are some positive constants $\displaystyle\varepsilon<\frac{h(\delta)}{2}$ and $\displaystyle\theta_0<\frac{\pi}{2}$ such that $\theta(r)\geqslant\theta_0$ for $r\in(0,\varepsilon).$

On the other hand, since $\ell$ is a ray in $C_0(Y)$ and $C_0(Z)$, by Lemma~\ref{ContainRay}, shrinking $\varepsilon$ if necessary, there are two $C^1$ definable curves $y\colon(0,\varepsilon)\to Y$ and $z\colon(0,\varepsilon)\to Z$ such that 
$$\|y(r)\|=\|z(r)\|=r,\ \ \left\|\frac{y(r)}{\|y(r)\|}-v\right\|\to 0\ \ \text{and}\ \ \left\|\frac{z(r)}{\|z(r)\|}-v\right\|\to 0 \ \text{ as } \ r\to 0^+.$$
So 
$$\frac{\|y(r)-rv\|}{r}=\left\|\frac{y(r)}{\|y(r)\|}-v\right\|\to 0 \ \ \text{and}\ \ \frac{\|z(r)-rv\|}{r}=\left\|\frac{z(r)}{\|z(r)\|}-v\right\|\to 0 \ \text{ as } \ r\to 0^+.$$
Consequently, $y(r)\in Y_r$ and $z(r)\in Z_r$ for any $r>0$ small enough. Fix $r>0$ such that \begin{equation}\label{F3}\displaystyle\frac{\|y(r)-rv\|}{r}<\frac{\delta\theta_0}{4\sqrt{2}}\ \quad \text{and}\ \quad \frac{\|z(r)-rv\|}{r}<\frac{\delta\theta_0}{4\sqrt{2}}.
\end{equation}
Let  
$$\gamma\colon[0,t_0)\to Y_r\times Z_r,\ t\mapsto(\alpha(t),\beta(t))$$ 
be the maximal trajectory of the vector field $-\frac{\nabla\rho}{\|\nabla\rho\|^2}$ on ${Y_r\times Z_r}$ with the initial condition $\gamma(0)=(y(r),z(r))$, where $\rho$ is the restriction of the function $\widetilde\rho$, defined in Lemma~\ref{Key}, on $Y_r\times Z_r$. The following facts are easy to verify: 
\begin{itemize}
\item [(a)] $\rho(\gamma(t))=\|y(r)-z(r)\|-t$ for $t\in[0,t_0)$ and so $t_0<\|y(r)-z(r)\|$; and
\item [(b)] $\displaystyle \|\nabla\rho(\gamma(s))\|\geqslant\frac{\theta_0}{\sqrt{2}}$ for $t\in[0,t_0)$. (This follows from Lemma~\ref{Key}(ii) and the assumption $\theta(r) \geqslant \theta_0>0.$)
\end{itemize}
Moreover we have
$$\begin{array}{llllllll}
\text{length}(\gamma) & = & \displaystyle\int_0^{t_0}\|\dot{\gamma}(t)\|dt & = & \displaystyle \int_0^{t_0}\frac{1}{\|\nabla\rho(\gamma(t))\|}dt\\
& \leqslant & \displaystyle\int_0^{t_0}\frac{\sqrt{2}}{\theta_0}dt & =& \displaystyle\frac{\sqrt{2}t_0}{\theta_0}\\
&<&\displaystyle\frac{\sqrt{2}\|y(r)-z(r)\|}{\theta_0} &\leqslant & \sqrt{2}\displaystyle\frac{\|y(r)-rv\|+\|z(r)-rv\|}{\theta_0}.
\end{array}$$
This, together with~\eqref{F3}, yields $\text{length}(\gamma)<\displaystyle\frac{r\delta}{2},$ which implies that the limit 
$$(\alpha_{t_0},\beta_{t_0})=\gamma_{t_0}:=\lim_{t\to t_0}\gamma(t)$$ exists. 
Clearly $\gamma_{t_0}\in \overline Y_r\times \overline Z_r$. 
In addition, we have
$$\begin{array}{llllllll}
\|\alpha_{t_0}-rv\|\leqslant\|\alpha_{t_0}-y(r)\|+\|y(r)-rv\|&\leqslant&\text{length}(\gamma)+\|y(r)-rv\| \\
&<&\displaystyle\frac{r\delta}{2}+\frac{r\delta\theta_0}{4\sqrt{2}}<r\delta.
\end{array}$$
Similarly, we also have $\|\beta_{t_0}-rv\|<r\delta.$ 
Furthermore, as $X$ is closed, it follows that 
$$\overline Y_r\setminus Y_r\subset {\mathbb S}_{r\delta}^{n-1}(rv)\ \text{ and }\ \overline Z_r\setminus Z_r\subset {\mathbb S}_{r\delta}^{n-1}(rv).$$
Therefore $\gamma_{t_0}\in Y_r\times Z_r$. 
Recall that $\gamma$ is maximal, so $\gamma_{t_0}$ is a critical point of $\rho.$ 
In light of Lemma~\ref{Key}, $\angle(T_{\alpha_{t_0}}Y_r,T_{\beta_{t_0}}Z_r)=0$. Consequently, $\theta(r)=0$ which contradicts the assumption. Therefore, \eqref{PT5} must hold and the theorem follows.
\end{proof} 

Let us make some preparation before proving Theorem~\ref{Singular}. First of all, we need the following technical lemma. 

\begin{lemma}\label{2sides}
Let $Y\subset \mathbb R^n$ be a definable set and $v\in Y.$ 
Assume that $\ell^*$ is a ray not in $C_vY$. 
Then, for any $R>0$, there exists a positive constant $\delta_R\leqslant\frac{1}{R}$ such that: 
\begin{enumerate}[{\rm (i)}]
\item for any $w\in C_v Y$ with $\|w\|\leqslant \delta_R$, we have $\displaystyle\dist(v+w,Y)<\frac{\|w\|}{R};$ 
\item for any $x\in Y$ with $\|x-v\|\leqslant \delta_R$, we have $\displaystyle\dist(x-v,C_v Y)<\frac{\|x-v\|}{R};$ 
\item for all $t\in(0,\delta_R)$, we have 
$$t\left(\sin\theta-\frac{1}{R}\right)<\dist(v+tp,Y)<t\left(\sin\theta+\frac{1}{R}\right),$$ 
where $p$ is the unit direction in $\ell^*$ and 
\begin{equation}\label{F3.1}
\theta:=\min\left\{\frac{\pi}{2},\min_{\ell'\subset C_v Y}\widehat{{\ell^*},\ell'}\right\}>0,
\end{equation} 
recall that $\widehat{{\ell^*},\ell'}$ is the angle between the rays $\ell^*$ and $\ell'$. 
\end{enumerate}
\end{lemma}

\begin{proof} 
(i) Suppose for contradiction that there is $R>0$ such that for any integer $k>0$, there is $w^k\in C_v Y$ with $\displaystyle\|w^k\|\leqslant\frac{1}{k}$ such that $\displaystyle\dist(v+w^k,Y)\geqslant \frac{\|w^k\|}{R}.$ 
Taking a subsequence if necessary, we may suppose that the sequence $\displaystyle\frac{w^k}{\|w^k\|}$ converges to a limit $w$. 
By Lemma~\ref{ContainRay}, there is a $C^1$ definable curve $\gamma:(0,\varepsilon)\to C_v Y\setminus\{0\}$ such that:
\begin{itemize}
\item [(a)] $\|\gamma(t)\|=t$ for $t\in(0,\varepsilon)$;
\item [(b)] $\displaystyle\dist(v+\gamma(t),Y)\geqslant \frac{\|\gamma(t)\|}{R}=\frac{t}{R}$ for $t\in(0,\varepsilon)$; and 
\item [(c)] $\displaystyle\lim_{t\to 0^+}\frac{\gamma(t)}{\|\gamma(t)\|}=w.$
\end{itemize}
Evidently $w\in C_v Y$. So by the definition of tangent cone and by Lemma~\ref{ContainRay}, there is a $C^1$ definable curve $\alpha:(0,\varepsilon')\to Y \setminus\{v\}$ such that:
\begin{itemize}
\item [(d)] $\|\alpha(t)-v\|=t$ for $t\in(0,\varepsilon')$; and
\item [(e)] $\displaystyle\lim_{t\to 0^+}\frac{\alpha(t)-v}{\|\alpha(t)-v\|}=w.$
\end{itemize}
Now we have
$$\begin{array}{lll}
\dist(v+\gamma(t),Y)&\leqslant&\|v+\gamma(t)-\alpha(t)\|\\ 
&\leqslant&\|\gamma(t)-tw\|+\|v-\alpha(t)+tw\|\\
&=&\displaystyle t\left(\left\|\frac{\gamma(t)}{\|\gamma(t)\|}-w\right\|+\left\|\frac{\alpha(t)-v}{\|\alpha(t)-v\|}-w\right\|\right).
\end{array}$$
This, together with Item (c) and Item (e), implies that $\displaystyle\lim_{t\to 0^+}\frac{\dist(v+\gamma(t),Y)}{t}=0$ which contradicts Item (b). Hence, Item (i) follows.

(ii) By contradiction, suppose that there is $R>0$ such that for any integer $k>0$, there is $x^k\in Y$ with $\displaystyle\|x^k-v\|\leqslant\frac{1}{k}$ such that $\displaystyle\dist(x^k-v,C_v Y)\geqslant \frac{\|x^k-v\|}{R},$ i.e., 
$$\dist\left(\frac{x^k-v}{\|x^k-v\|},C_v Y\right)\geqslant \frac{1}{R}.$$ 
Taking a subsequence if necessary, we may suppose that the sequence $\displaystyle\frac{x^k-v}{\|x^k-v\|}$ converges to a limit $w$. 
Then, clearly, $w\in C_v Y$ on one hand and we have $\displaystyle\dist(w,C_vY)\geqslant \frac{1}{R}$ on the other hand. 
This is a contradiction and Item (ii) follows.

(iii) By Item (i), for each $R$, there exists $\delta_R>0$ such that for any $w\in C_v Y$ with $\|w\|\leqslant \delta_R$, we have $\displaystyle\dist(v+w,Y)<\frac{\|w\|}{R}.$ 
Since $C_vY$ is closed, there is $\widetilde w\in C_v Y$ such that the distance function $\dist(tp,C_v Y)$ is attained, i.e., 
$$\dist(tp,C_v Y)=\|tp-\widetilde w\|=t\sin\theta.$$ 
Then clearly $\|\widetilde w\|=t\cos\theta<t<\delta_R$. 
In addition,  
$$\|tp-\widetilde w\|-\dist(v+\widetilde w,Y)\leqslant\dist(v+tp,Y)\leqslant\|tp-\widetilde w\|+\dist(v+\widetilde w,Y).$$
Thus, in view of Item (i), we get
$$t\sin\theta-\frac{t\cos\theta}{R}\leqslant\dist(v+tp,Y)\leqslant t\sin\theta+\frac{t\cos\theta}{R},$$
which implies Item (iii).
\end{proof}

The following lemma is the key to prove Theorem~\ref{Singular}.
\begin{lemma}\label{perpendicular} Let $Y\subset \mathbb R^n$ be a closed definable set and $v\in Y$. 
Let $Y_k$ be a sequence of closed definable sets of pure dimension $d>0$ for all integer $k>0$ such that 
$$Y=\lim_{k\to +\infty}Y_k\ \text{ and }\ v\not\in \lim_{k\to +\infty}(Y_k)_{sing}.$$ 
Assume that $y^k\in Y_{k}\setminus (Y_k)_{sing}$ is a sequence tending to $v$ such that the limit $\displaystyle Q:=\lim_{k\to +\infty}T_{y^k}Y_{k}$ exists and $C_vY\subsetneq Q.$
Then there is a subsequence $\{k_1,k_2,k_3\dots\}$ of the sequence $\{1,2,3,\dots\}$ and a sequence $x^{l}\in Y_{k_l}\setminus (Y_{k_l})_{sing}$ tending to $v$ such that the limit $\displaystyle P:=\lim_{l\to +\infty}T_{x^{l}}Y_{{k_l}}$ exists and $\displaystyle\angle(P,Q)=\frac{\pi}{2}.$

\end{lemma}
\begin{proof}
The construction in the proof is described in Figure 2 below. 

\begin{tikzpicture}

\draw [blue] (-8,10) -- (6,10);
\draw [blue] (-8,10) -- (-5,16);
\draw [line width=1.1pt,red,dashed] (6,12) arc (0:180:6 and 2.1);
\draw [line width=1.1pt,red] (6,12) arc (0:-30.5:6 and 2);
\draw [line width=1.1pt,red] (-6,12) arc (-180:-114.5:6 and 2);
\draw [line width=1.1pt,red] (0,12) arc (-25:-61:6 and 4);
\draw [line width=1.1pt,red] (0,12) arc (-133:-79.5:6 and 4);

\draw [violet,dashed] (6,12) arc (0:180:6 and 2);
\draw [violet] (6,12) arc (0:-60:6 and 2);
\draw [violet] (-6,12) arc (-180:-104.5:6 and 2);
\draw [violet] (-1.5,10.0635) -- (0,12);
\draw [violet] (3,10.268) -- (0,12);

\draw [orange,dashed] (0,14.2) arc (90:177:6 and 2);
\draw [orange,dashed] (0,14.2) arc (90:3:6 and 2);
\draw [orange] (0,14.5) arc (90:175:6 and 2);
\draw [orange] (0,14.5) arc (90:5:6 and 2);
\draw [orange] (6,12.3) arc (0:-30:6 and 2);
\draw [orange] (5.97,12.7) arc (0:-32.2:6 and 2);
\draw [orange] (-6,12.3) arc (-180:-120:6 and 2);
\draw [orange] (-5.98,12.7) arc (-180:-119.7:6 and 2);
\draw [orange] (-3,10.571) arc (-95:85:0.2);  
\draw [orange] (5.05,11.64) arc (110:300:0.187); 
\draw [orange] (-2.95,10.97) arc (-61:-30:7 and 4);
\draw [orange] (-2.8,10.77) arc (-61:-30:7 and 4);
\draw [orange] (-2.95,10.57) arc (-62:-29:7 and 4);
\draw [orange] (5.05,11.64) arc (-80:-128:6 and 4);
\draw [orange] (4.92,11.446) arc (-80:-128:6 and 4);
\draw [orange] (5.13,11.28) arc (-79.2:-129:6 and 4);
\draw [orange] (0.32,12.42) arc (45:127:0.46);
\draw [orange] (0.19,12.23) arc (45:125:0.26);
\draw [orange] (0.24,12.09) arc (50:110:0.37);

\node [red,below] at (0,14) {$Y$};
\node [violet] at (4.4,10.2) {$C_v Y$};
\node [right] at (-0.1,11.7) {$v$};
\node [below,blue] at (-7,10) {$\{v\}+Q$};
\node [right,blue] at (0,7) {$\ell^*$};
\node [right,blue] at (0,8) {$p$};
\node [left] at (0.1,11.4) {$\theta$};
\node [orange,above] at (0,14.5) {$Y_{k_l}$};
\node [right] at (0,10.5) {$v+t_lp$};
\node [below] at (-4.2,6.5) {$q^l$};
\node [above] at (-4.5,17) {$x^{l}$};
\node [below] at (-4.7,7) {$\overline x^{l}$};

\draw (0,11.6) arc (-90:-120:0.4);
\draw [blue] (0,12) -- (0,10);
\draw [blue] (0,10) -- (0,6.7);
\draw [blue,->] (0,10) -- (0,8);
\draw [violet] (0,10.5) -- (-1.1,10.9);
\draw [orange] (0,10.5) -- (-1.4,11.3);
\filldraw (-1.1,10.9) circle (1.5pt);
\filldraw (-1.4,11.3) circle (1.5pt);
\filldraw (-1.45,11) circle (1.5pt);
\filldraw (0,10.5) circle (1.5pt);
\filldraw (0,12) circle (1.5pt);
\draw [->] (-4,6.5) -- (-1.15,10.8) ;
\draw [->] (-4.5,17) -- (-1.45,11.4) ;
\draw [->] (-4.5,7) -- (-1.5,10.9) ;

\node at (-0.5,5.5) {$Y$ is the region with red boundary};
\node at (-0.5,5) {$C_v Y$ is the region with violet boundary};
\end{tikzpicture}
$$\textrm{Figure 2.}$$

By the assumption, there exists a ray $\ell^*$ in $Q$ which is not a ray in $C_vY$. 
Let $\theta>0$ be the constant given by~\eqref{F3.1}. 
For each integer $l>0$, let $\delta_l\in(0,\frac{1}{l}]$ be the constant given by Lemma~\ref{2sides}. By the assumption $\displaystyle\lim_{k\to +\infty}Y_k=Y$, for each integer $l>0$, there is an integer $k_l>0$ 
such that 
\begin{equation}\label{F5}
\dist_{\mathcal H} (Y_{k_l},Y)<\frac{\delta_l}{l},
\end{equation}
where $\dist_{\mathcal H}(\cdot,\cdot)$ still denotes the Hausdorff distance. 
Let $\displaystyle t_l=\frac{\delta_l}{3}$. Then for $l$ large enough, we have
\begin{equation}\label{F6}
\begin{array}{lllll}
\dist(v+t_lp,Y_{k_l})&\geqslant&\dist(v+t_lp,Y)-\dist(Y,Y_{k_l})\\
&>&\displaystyle t_l\left(\sin\theta-\frac{1}{l}\right)-\frac{\delta_l}{l}\\
&=&\displaystyle t_l\left(\sin\theta-\frac{1}{l}\right)-\frac{3t_l}{l}
=\displaystyle t_l\left(\sin\theta-\frac{4}{l}\right)>0, 
\end{array}
\end{equation}
where the second inequality follows from Lemma~\ref{2sides}(iii) and~\eqref{F5}. Hence $v+t_lp\not\in  Y_{k_l}$. 
By the closedness of $Y$ and $Y_{k_l}$, let $x^{l}\in  Y_{k_l}$ and $q^{l}\in Y$ be respectively points where the distance functions $\dist(v+t_l p, Y_{k_l})$ and $\dist(v+t_lp,Y)$ are attained. 
Then for $l$ large enough, 
we have
\begin{equation}\label{F6-}
\begin{array}{lllll}
\|x^{l}-v\|&\leqslant&\|x^{l}-v-t_l p\|+t_l\\
&=&\dist(v+t_l p,Y_{k_l})+t_l\\
&\leqslant&\dist(v+t_l p,q^{l})+\dist(q^{l},Y_{k_l})+t_l\\
&\leqslant&\displaystyle\dist(v+t_l p,Y)+\sup_{x\in Y}\dist(x,Y_{k_l})+t_l\\ 
&\leqslant&\dist(v+t_l p,Y)+\dist_{\mathcal H}(Y,Y_{k_l})+t_l\\ 
&<&\displaystyle\frac{\delta_l}{3}\left(\sin\theta+\frac{1}{l}\right)+\frac{\delta_l}{l}+\frac{\delta_l}{3}\leqslant\delta_l\left(\frac{2}{3}+\frac{4}{3l}\right) <\delta_l\leqslant\frac{1}{l}, 
\end{array}
\end{equation}
where the fifth inequality follows from Lemma~\ref{2sides}(iii) and~\eqref{F5}.
Consequently \begin{equation}\label{xv}\displaystyle\lim_{l\to+\infty}x^{l}=v.\end{equation}

Let $z^l\in Y$ be such that $$\dist(x^{l},Y)=\|x^{l}-z^l\|.$$ 
Obviously $\dist(x^{l},Y)\leqslant\|x^{l}-v\|$, so in view of~\eqref{F6-}, we have
\begin{equation}\label{F6=}
\|z^l-v\|\leqslant\|z^l-x^{l}\|+\|x^{l}-v\|\leqslant 2\|x^{l}-v\|<2\delta_l.
\end{equation}
Then 
\begin{equation}\label{F6-=}
\begin{array}{lllll}
\dist(x^{l},\{v\}+Q)&\leqslant&\dist(x^{l},\{v\}+C_v Y)\\
&\leqslant&\dist(x^{l},Y)+\dist(z^l,\{v\}+C_v Y)\\
&=&\dist(x^{l},Y)+\dist(z^l-v,C_v Y)\\
&\leqslant&\dist_{\mathcal H} (Y_{k_l},Y)+\dist(z^l-v,C_v Y)\\
&\leqslant&\displaystyle\frac{\delta_l}{l}+\frac{\|z^l-v\|}{l}<\frac{3\delta_l}{l},
\end{array}
\end{equation}
where the first inequality follows from the assumption $C_v Y\subset Q$, the forth inequality follows from Lemma~\ref{2sides}(ii) and~\eqref{F5} while the fifth one follows from~\eqref{F6=}. 
Let $\overline x^{l}$ be the orthogonal projection of $x^{l}$ on $\{v\}+Q$, then by~\eqref{F6} and~\eqref{F6-=}, we have
\begin{equation}\label{F7}
\frac{\|x^{l}-\overline x^{l}\|}{\|v+t_l p-x^{l}\|}<\frac{3\delta_l}{l t_l\Big(\sin\theta-\frac{4}{l}\Big)}=\frac{9}{l\Big(\sin\theta-\frac{4}{l}\Big)}\to 0\ \text{ as }\ l\to+\infty. 
\end{equation}
By taking a subsequence if necessary, we may assume that the sequence $\displaystyle\frac{v+t_l p-x^{l}}{\|v+t_l p-x^{l}\|}$ converges to a limit $w$. 
Then from~\eqref{F7}, it is not hard to check that $$\lim_{l\to+\infty}\frac{v+t_l p-\overline x^{l}}{\|v+t_l p-\overline x^{l}\|}=w.$$ 
Since $v+t_l p,\overline x^{l}\in \{v\}+ Q,$ we have
$$\angle(v+t_l p-\overline x^{l},Q)=\angle(v+t_l p-\overline x^{l},\{v\}+Q)=0.$$ 
Hence $\angle(w, Q)=0$. 
By the assumption $\displaystyle v\not\in \lim_{l\to +\infty}(Y_{k_l})_{sing}$ and~\eqref{xv}, it is clear that $x^{l}$ is not a singular point of $Y_{k_l}$ for $l$ large enough. 
Taking a subsequence if necessary, we can assume that there exists the limit $$P:=\lim_{l\to+\infty}T_{x^{l}}Y_{k_l}.$$ 
Observe that $v+t_l p- x^{l}$ is perpendicular to $T_{x^{l}}Y_{k_l}$ as $x^{l}$ is a point where the distance function $\dist(v+t_l p, Y_{k_l})$ is attained, so by taking limit as $l\to+\infty$, we get $\displaystyle\angle(w,P)=\frac{\pi}{2}$. 
Hence $\displaystyle\angle(P, Q)=\frac{\pi}{2}$ and the lemma follows. 
\end{proof}

We finish the section by giving the proof of Theorem \ref{Singular}.

\begin{proof}[Proof of Theorem \ref{Singular}] 
Let $Q\in\mathbb G(d,n)$ be a tangent limit of $X$ along $\ell$. There are two cases to be considered. 

\subsubsection*{Case $1$: $C_v\mathcal C\not\subset Q$, i.e., $C_v\mathcal C\setminus Q\ne \emptyset$} 

Let $\widetilde\ell$ be a ray in $C_v\mathcal C\setminus Q$.
As $\mathcal C_{sing}$ is nowhere dense in $\mathcal C$, 
it is not hard to see that there is a sequence $v^k\in \mathcal C\setminus \mathcal C_{sing}$ such that 
$$v^k\to v,\ \displaystyle\frac{v^k-v}{\|v^k-v\|}\to u\in\widetilde\ell\ \text{ and }\ T_{v^k}\mathcal C\to \mathcal P\ \text{as}\ k\to +\infty.$$ 
In light of Lemma~\ref{ContainRay}, $\widetilde\ell\subset \mathcal P$. 
Since $\widetilde\ell\not\subset Q$, obviously $\mathcal P\not\subset Q$. 
Denote by $\ell_k$ the ray in $\mathcal C$ through $v^k$. 
It is clear that 
$$\ell_k\subset \mathcal C\setminus \mathcal C_{sing} \text{ for any } k \ \text{ and }\ \ell_k\to\ell:=\R_+v\ \text{as}\ k\to +\infty.$$
By Corollary~\ref{TP}, there exists $P\in\mathcal N_\ell$ such that $\mathcal P\subset P.$ As $\mathcal P\not\subset Q$, it follows that $P\ne Q$ and so $\# (\mathcal N_\ell)>1$.

\subsubsection*{Case $2$: $C_v\mathcal C\subset Q$.} 
By the assumption, we have $C_v\mathcal C\ne  Q.$
We will show that there is a tangent limit $P$ of $X$ along $\ell$ such that $\displaystyle\angle(P,Q)=\frac{\pi}{2}$, which yields the theorem. 
By the definition, there is a sequence $z^k\in X\setminus X_{sing}$ and a sequence $t_k\in(0,+\infty)$ such that 
$$\lim_{k\to +\infty}z^k=0,\ \lim_{k\to +\infty}t_kz^k=v\ \text{ and }\ Q=\lim_{k\to +\infty}T_{z^k}X.$$ 
Since $X$ is of pure dimension $d$ at $0$ and since $t_k\to +\infty$ as $k\to +\infty$, there is $\delta\in(0,1)$ such that $X\cap \mathring{\mathbb B}^n_{\frac{\delta}{t_k}}(\frac{v}{t_k})$ is of pure dimension $d$ for $k$ large enough. For such $k$, set
$$X_k:={t_k} X=\left\{{t_k}x :\ x\in X\right\},\ \ Y_k:=\overline{X_k\cap \mathring{\mathbb B}^n_{\delta}(v)}\ \text{ and } \ Y:=\lim_{k\to +\infty}Y_k.$$

We are going to apply Lemma~\ref{perpendicular}, so we need to verify the conditions required by this lemma.
Set $y^k=t_k x^k$. It is clear that $Y_k$ and $Y$ are closed, $v\in Y$, $y^k\in Y_k\setminus (Y_k)_{sing}$ and $\displaystyle Q=\lim_{k\to +\infty}T_{y^k}Y_k$.

We will show that $\displaystyle v\not\in \lim_{k\to +\infty}(Y_k)_{sing}.$ 
By contradiction, suppose that $\displaystyle v\in \lim_{k\to +\infty}(Y_k)_{sing}.$
Then there is a sequence $w^k\in (Y_k)_{sing}$ tending to $v$. 
Clearly $w^k\in \mathring{\mathbb B}^n_{\delta}(v)$ for $k$ large enough.
This and the condition $w^k\in (Y_k)_{sing}$ implies that $w^k\in (\overline X_k)_{sing}$, i.e., $\displaystyle u^k:=\frac{w^k}{t_k}\in (\overline X)_{sing}$.
As $t_k\to +\infty,$ we have $u^k\to 0$ as $k\to +\infty$. 
Moreover, it is clear that $t_k u^k\to v$. 
Therefore $v\in \mathcal C'$. 
This contradiction implies that $\displaystyle v\not\in \lim_{k\to +\infty}(Y_k)_{sing}.$

Next, for $k$ large enough, we must have that $X\cap \mathring{\mathbb B}^n_{\frac{\delta}{t_k}}(\frac{v}{t_k})$ is of pure dimension $d$. 
Therefore $X_k\cap \mathring{\mathbb B}^n_{\delta}(v)$ is also of pure dimension $d$ as it is the image of $X\cap \mathring{\mathbb B}^n_{\frac{\delta}{t_k}}(\frac{v}{t_k})$ by the linear isomorphism 
$$\mathbb R^n\to\mathbb R^n,\ x\mapsto t_k x.$$ 
Consequently $Y_k=\overline{X_k\cap \mathring{\mathbb B}^n_{\delta}(v)}$ is of pure dimension $d$. 

In order to apply Lemma~\ref{perpendicular}, it remains to prove that $C_v Y\subsetneq Q.$ 
For this, it is sufficient to show that $C_v\mathcal C=C_v Y$.
Let $u\in \mathcal C\cap \mathring{\mathbb B}^n_{\delta}(v).$ Clearly $u\ne 0$ by the choice of $\delta$. 
In view of Lemma~\ref{ContainRay}, there is a $C^1$ definable curve $\gamma:(0,\varepsilon)\to X\setminus X_{sing}$ such that 
$$\|\gamma(r)\|=r\ \text{for}\ r\in(0,\varepsilon)\ \text{ and }\ \lim_{r\to 0^+}\frac{\gamma(r)}{r}= u.$$
Thus it is clear that, for $k>0$ large enough, $t_k\gamma(\frac{1}{t_k})\in \mathring{\mathbb B}^n_{\delta}(v)$ \text{and so} $t_k\gamma(\frac{1}{t_k})\in Y_k.$
Consequently $u\in Y$ and we get $\mathcal C\cap \mathring{\mathbb B}^n_{\delta}(v)\subset Y.$ On the other hand, it is easy to see that $Y\subset\mathcal C\cap{\mathbb B}^n_{\delta}(v).$ Therefore $\mathcal C\cap \mathring{\mathbb B}^n_{\delta}(v)=Y\cap \mathring{\mathbb B}^n_{\delta}(v)$, which implies that $C_v\mathcal C=C_v Y.$

Now in light of Lemma~\ref{perpendicular}, there is a subsequence $k_l$ of $\{1,2,\dots\}$ and a sequence $x^{l}\in Y_{k_l}\setminus(Y_{k_l})_{sing}$ tending to $v$ as $l\to +\infty$ such that the limit $\displaystyle P:=\lim_{l\to +\infty}T_{x^{l}}Y_{k_l}$ exists and $\displaystyle\angle(P,Q)=\frac{\pi}{2}.$
Let $\displaystyle\widetilde x^l:=\frac{x^l}{t_{k_l}}\in \overline X.$
Clearly $\widetilde x^l\to 0$, $t_{k_l}\widetilde x^l=x^l\to v$ as $l\to+\infty$ and $\widetilde x^l$ is a non singular point of $\overline X.$
In addition, since $T_{\widetilde x^{l}}\overline X$ and $T_{x^{l}}Y_{k_l}$ determine the same plane in the Grassmannian $\mathbb G(d,n)$, we get $\displaystyle\lim_{l\to +\infty} T_{\widetilde x^{l}}\overline X=P,$ i.e., $P$ belongs to the Nash fiber of $\overline X$ along $\ell$. In view of Lemma~\ref{NotEmpty}, we also have $P\in\mathcal N_\ell$. The theorem follows.
\end{proof}

The following corollary follows immediately from the proof of Theorem~\ref{Singular}.
\begin{corollary} Let $X\subset\mathbb R^n$ be a definable set of pure dimension $d$ at $0$. If $\dim \mathcal C<d$, then $\mathcal E=\mathcal C\setminus \mathcal C'$.
\end{corollary}

\section{Remarks and examples}\label{Remarks}

In this section we give some remarks and examples concerning the results presented in the paper.

\begin{remark}{\rm 
\begin{enumerate}[{\rm (i)}]
\item It is possible that $\dim \mathcal E=\dim\mathcal C$.
In addition, under the assumptions of Theorem \ref{Singular}, it does not necessarily hold that $v\in \mathcal E'$. These will be seen in Example~\ref{E=C}.
\item Theorem \ref{Connected} does not necessarily hold if we replace $\mathcal E$ by $\mathcal E'$. In fact, for a ray $\ell$ in $\mathcal C'$, the Nash fiber $\mathcal N_\ell$ along $\ell$ is not necessary connected as shown in Example~\ref{2E'}. 
\item If $\dim \mathcal C=\dim X$, a ray $\ell\subset \mathcal C_{sing}$ does not necessary belong to $\mathcal E$. This is illustrated in Example~\ref{E1}.
\item If $X$ is not closed, it is worth noting that we need to remove from $\mathcal E$ the rays in $\mathcal C'=C_0(\overline X)_{sing}$, not only the rays in $C_0 X_{sing}$. Precisely, Theorem~\ref{Connected} may not hold if we set 
$$\mathcal E=\{\ell\subset \mathcal C\setminus C_0 X_{sing}:\ \#(\mathcal N_\ell) > 1\}.$$ 
An illustration is given in Example~\ref{NotSBX}.
\item Theorem~\ref{Connected} does not hold, in general, for definable sets of codimension greater than $1$ as shown in Example~\ref{Codim2}.
\end{enumerate}
}\end{remark}

\begin{example}{\rm Consider the Whitney umbrella 
$$X:=\{(x,y,x)\in\R^3:\ x^2-y^2z=0\}.$$ 
It is not hard to see that 
$$X_{sing}=\{x=y=0\}\ \text{ and }\ \mathcal C=\{x=0,z\geqslant 0\}.$$ 
We will show that the rays $\R_+ (0,1,0)$ and $\R_+ (0,-1,0)$ belong to $\mathcal E$ by computing the Nash fibers along these rays (in fact, this is straightforward in view of Theorem~O'Shea--Wilson or Theorem~\ref{Singular}); moreover, 
\begin{equation}\label{51}\begin{array}{lll}
\mathcal N_{\R_+ (0,1,0)}&=&\mathcal N_{\R_+ (0,-1,0)}\\
&=&\{P\in\mathbb G(2,3):\ P \text{ contains the axis } Oy \}\\
&=&
\left\{\begin{array}{lll}
P\in\mathbb G(2,3):\ P=
\left\{\begin{array}{lll}(w_1,w_2,w_3): & aw_1+bw_3=0,\\ 
&a^2+b^2\ne 0,\ b\geqslant 0
\end{array}\right\}
\end{array}\right\}.
\end{array}
\end{equation}
Assume that $\ell=\R_+ (0,1,0)$. The case $\ell=\R_+ (0,-1,0)$ is similar.
Set $$f:=x^2-y^2z.$$ 
For $(x,y,z)\in X$, we have 
$$\nabla f(x,y,z)=(2x,-2yz,-y^2).$$
Set $$A(x,y,z):=\frac{\nabla f(x,y,z)}{\|\nabla f(x,y,z)\|}.$$
Let $\gamma\colon(0,\epsilon)\to X\setminus X_{sing}$ be an analytic curve such that $\gamma(t)\to 0$ and $\displaystyle\frac{\gamma(t)}{\|\gamma(t)\|}\to (0,1,0)$ as $t\to 0$. 

If the curve $\gamma$ lies in the axis $Oy$, then $A(\gamma(t))=(0,0,-1)$. 
Thus 
\begin{equation}\label{51-1}\mathcal N_\ell\ni(0,0,-1)^\perp=\{w_1,w_2,w_3)\in\R^3: w_3=0\}.\end{equation}
Now assume that $\gamma$ does not intersects the axis $Oy$.
Write $$\gamma(t)=(x(t),y(t),z(t))=(x_0t^\alpha+\cdots,y_0t^\beta+\cdots,z_0t^\gamma+\cdots).$$ 
Clearly $y_0> 0$ and $\beta>0$. 
Moreover, as $\gamma$ does not intersect the axis $Oy$, it follows that $z(t)> 0$ for all $t$. 
Hence $z_0>0$ and $\gamma>0$. 
Consequently $x_0\ne 0$ and $\alpha>0.$ 
As $$z(t)=\frac{x^2(t)}{y^2(t)}=\frac{x_0^2}{y_0^2}t^{2\alpha-2\beta}+\dots,$$ we get
$$\gamma(t)=\left(x_0t^\alpha,y_0t^\beta,\frac{x_0^2}{y_0^2}t^{2\alpha-2\beta}\right)+\cdots\ \text{and}$$ 
$$\nabla(\gamma(t))=\left(2x_0t^\alpha,-2\frac{x_0^2}{y_0}t^{2\alpha-\beta},-{y_0^2}t^{2\beta}\right)+\cdots$$ 
Observe that $\gamma=2\alpha-2\beta$, so $\alpha>\beta.$
Assume that $\alpha=2\beta$, then it is not hard to verify that 
$$\lim_{t\to 0}\frac{\gamma(t)}{\|\gamma(t)\|}=(0,1,0)\ \text{ and }\ \lim_{t\to 0}A(\gamma(t))=\frac{(2x_0,0,-y_0^2)}{\sqrt{4x_0^2+y_0^4}}.$$
For $a\ne 0$ and $b>0$, set $x_0=-\frac{a}{2}$ and $y_0=\sqrt{b}$.
Then 
\begin{equation}\label{51-2}
\begin{array}{lll}\displaystyle\mathcal N_\ell\ni \left(\frac{(2x_0,0,-y_0^2)}{\sqrt{4x_0^2+y_0^4}}\right)^\perp&=&(2x_0,0,-y_0^2)^\perp\\
&=&(-a,0,-b)^\perp=(a,0,b)^\perp\\
&=&\{(w_1,w_2,w_3): aw_1+bw_3=0\}.
\end{array}
\end{equation}
Now suppose that $\frac{3\beta}{2}<\alpha<2\beta.$ By simple computations, we have 
$$\lim_{t\to 0}\frac{\gamma(t)}{\|\gamma(t)\|}=(0,1,0)\ \text{ and }\ \lim_{t\to 0}A(\gamma(t))=(1,0,0),$$
i.e., $\mathcal N_\ell\ni(1,0,0)^\perp=\{(w_1,w_2,w_3):\ w_1=0\}.$
Combining this with~\eqref{51-1} and~\eqref{51-2} yields~\eqref{51}.

For any ray $\ell$ different from $\R_+ (0,1,0)$ and $\R_+ (0,-1,0)$, it can be verified that $\mathcal N_\ell$ contains only one element given by $\{(w_1,w_2,w_3):\ w_1=0\}.$ So $\ell\not\in \mathcal E.$
}\end{example}

\begin{example}\label{E=C}{\rm  Let $X:=\{(x,y,x)\in\R^3:\ x^2+y^2=z^3\}$. Then $\mathcal C=\{x=y=0,\ z\geqslant 0\}.$ Let $\ell:=\R_+(0,0,1)$. We have $\mathcal C=\ell\cup\{0\}$ and $\mathcal N_\ell=\{P\in \mathbb G(2,3):\ \ell\subset P\}$. Consequently $\dim\mathcal E=\dim\mathcal C$. }
\end{example}

\begin{example}\label{2E'}{\rm Let $X:=X_1\cup X_2$, where
$$X_1:=\{(x,y,x)\in\R^3:\ z^3\geqslant x^2,\ y=0\},$$
$$X_2:=\{(x,y,x)\in\R^3:\ x=0\}.$$
It is not hard to check that $\mathcal C=\{x=0\}$ and $\mathcal C'=\{x=y=0\}$. 
So $\mathcal C_{sing}=\emptyset.$ 
Consider the ray $\ell=\R_+(0,0,1)\subset \mathcal C'$. Clearly $\mathcal N_\ell$ is disconnected since it contains two elements which are determined respectively by $\{(w_1,w_2,w_3):\ w_1=0\}$ and $\{(w_1,w_2,w_3):\ w_2=0\}$.
}\end{example}

\begin{example}\label{E1}{\rm Let $X:=X_1\cup X_2$, where
$$X_1:=\{(x,y,x)\in\R^3:\ f_1(x,y,z):=x^2+(y-z)^2+z^4-z^2=0, z\geqslant 0\},$$ 
$$X_2:=\{(x,y,x)\in\R^3:\ f_2(x,y,z):=x^2+(y+z)^2+z^4-z^2=0, z\geqslant 0\}.$$
Clearly,
$$\mathcal C=\{(x^2+(y-z)^2-z^2)(x^2+(y+z)^2-z^2)=0,z\geqslant 0\},$$
which is the union of two cones tangent to each other along the ray $\ell=\R_+(0,0,1)$. Therefore, $\ell\subset \mathcal C_{sing}$. 
On the other hand, let $a^k:=(x_k,y_k,z_k)\in X\setminus\{(0,0,0)\}$ be any sequence such that $a^k\to 0$ and $\displaystyle\frac{a^k}{\|a^k\|}\to(0,0,1)$, so $\displaystyle\frac{x_k}{z_k}\to 0$ and $\displaystyle\frac{y_k}{z_k}\to 0.$ 
Without loss of generality, suppose that $a^k\in X_1$ for all $k.$ We have 
$$\nabla f_1(a^k)=(2x_k,2y_k-2z_k,-2y_k+4z_k^3)=2z_k\left(\frac{x_k}{z_k},\frac{y_k}{z_k}-1,-\frac{y_k}{z_k}+2z_k^2\right).$$
Consequently, $\frac{\nabla f_1(a^k)}{\|\nabla f_1(a^k)\|}\to (0,-1,0)$, which implies that $T_{a^k}X$ tends to the plane $Oxz.$ Hence $\mathcal N_\ell$ contains only one element which is the plane $Oxy$, so $\ell\not\subset \mathcal E$.}
\end{example}

\begin{example}\label{NotSBX}{\rm Set 
$$X:=\{(x,y,z)\in\R^3:\ x=0,\ y\ne 0\}\cup\{(x,y,z)\in\R^3:\ y=0,\ x\ne 0\}.$$
Note that $X_{sing}=\emptyset$, so $C_0 X_{sing}=\emptyset$. Obviously,
$$\mathcal C=\{x=0\}\cup\{y=0\},$$
Moreover, 
$$\ell_1:=\R_+(0,0,1)\subset \mathcal C\setminus C_0 X_{sing},\ \ell_2:=\R_+(0,0,-1)\subset \mathcal C\setminus C_0 X_{sing}$$
and 
$$\mathcal N_{\ell_1}=\mathcal N_{\ell_2}=\{P_1,P_2\},$$ 
where 
$$P_1=\{(w_1,w_2,w_3):\ w_1=0\}\ \text{ and }\ P_2=\{(w_1,w_2,w_3):\ w_2=0\}.$$ 
So $\mathcal N_{\ell_1}$ and $\mathcal N_{\ell_2}$ are disconnected.
}\end{example}

\begin{example} \label{Codim2}{\rm
Let $X:=X_1\cup X_2,$ where
$$\begin{array}{lll}
X_1&:=&\{(x,y,z,t)\in\mathbb R^4:\ x=y=0\}\\
X_2&:=&\{(x,y,z,t)\in\mathbb R^4:\ z=0,\ x^2+y^2=t^3\}.
\end{array}$$
It is clear that $X\setminus\{0\}$ is not singular, $\dim X=2$, and that
$$C_0(X_1)=X_1\ \ \text{and}\ \ C_0(X_2)=\{x=y=z=0\}.$$
So $C_0(X)=\{x=y=0\}.$ Let 
$$\ell:=\R_+(0,0,0,1)\subset C_0(X_1)\cap C_0(X_2)\subset \mathcal C.$$ 
Denote by $\mathcal N_\ell(X_i)$ the set of tangent limits of $X_i\ (i=1,2)$ along $\ell$. It is clear that $\mathcal N_\ell(X_1)$ has only one element given by 
$$P=\{(w_1,w_2,w_3,w_4):\ w_1=w_2=0\}.$$ 
On the other hand, for any $Q\in \mathcal N_\ell(X_2),$ its equation is given by 
$$Q=\{(w_1,w_2,w_3,w_4):\ aw_1+bw_2=0,w_3=0\}$$ 
with $a^2+b^2\ne 0.$ So we have $\displaystyle\angle(P,Q)=\frac{\pi}{2}.$ Consequently, the set of tangent limits of $X$ along $\ell$, given by $\mathcal N_\ell(X)=\mathcal N_\ell(X_1)\cup \mathcal N_\ell(X_2)$, is disconnected. Obviously, $\ell$ is a ray in $\mathcal E$ and $\dim \mathcal N_\ell=1$.

Now let $$X_3:=\{(x,y,z,t)\in\mathbb R^4:\ z=0,\ x^2=t^3\}$$ 
and let $X':=X_1\cup X_3.$ 
Then $(X')_{sing}=(X_3)_{sing}=\{x=z=t=0\}$, $\dim X'=2$ and we have 
$$C_0(X')=\{x=y=0\}\cup\{x=z=0,t\geqslant 0\}.$$ Let $\ell:=\R_+(0,0,0,1)\subset (C_0(X'))_{sing}.$ Clearly, $\ell$ is not tangent to $(X')_{sing}$ at $0$. Note that $\mathcal N_\ell(X_3)$ contains only one element given by 
$$R=\{(w_1,w_2,w_3,w_4):\ w_1=w_3=0\},$$ 
so $R\ne P$. Therefore $\#(\mathcal N_\ell)=2$ and $\ell$ is a ray in $\mathcal E$.
}

\end{example}
\noindent\textbf{Question.} If $X$ is a definable set of pure dimension $d>0$ at the origin $0\in\mathbb R^n$ and $\ell\subset \mathcal E\setminus (\mathcal C'\cup \mathcal C_{sing}),$ then $\dim \mathcal N_\ell\geqslant 1$?

\noindent\textbf{Acknowledgment.} A great part of this work was performed while the first author visited the laboratory LAMA $-$ Universit\'e Savoie Mont Blanc $-$ CNRS research unit number 5127 by benefiting a ``poste rouge'' of the CNRS. The first author would like to thank the laboratory, INSMI and LIA Formath Vietnam (CNRS) for hospitality and support. This research was also partially performed while the first and third authors visited  Vietnam Institute for Advanced Study in Mathematics (VIASM). The first and third authors would like to thank the Institute for hospitality and support. We also would like to thank Krzysztof Kurdyka and Vincent Grandjean for many helpful discussions during the preparation of the paper.


\end{document}